\documentclass{amsart}

\usepackage{amsthm,amsfonts,amsmath,amssymb,mathrsfs,tensor}
\usepackage[all]{xy}
\usepackage{pstricks}

\newtheorem{thm}{Theorem}[subsection]
\newtheorem{lem}[thm]{Lemma}
\newtheorem{prop}[thm]{Proposition}

\newtheorem{defn}[thm]{Definition}

\newenvironment{rmk}{\refstepcounter{thm} \medskip \noindent {\bf  Remark \arabic{section}.\arabic{subsection}.\arabic{thm}.\,}}{\hfill\mbox{}\bigskip}

\newtheorem{thmint}{Theorem}

\newtheorem{conjint}[thmint]{Conjecture}

\newcounter{num}

\newenvironment{thmlist}{\begin{list}{(\roman{num})}{\usecounter{num}\setlength{\leftmargin}{25pt}
\setlength{\itemindent}{0pt}\setlength{\labelwidth}{20pt}\setlength{\labelsep}{5pt}\setlength{\itemsep}{0in}}}{\end{list}}

\def\t{\mathfrak{t}}

\def\TT{\mathbb{T}}
\def\PP{\mathbb{P}}
\def\R{\mathbb{R}}
\def\C{\mathbb{C}}
\def\N{\mathbb{N}}
\def\Z{\mathbb{X}}
\def\Z{\mathbb{Z}}

\def\i{\sqrt{-1}}

\def\om{\omega}

\def\del{\partial}
\def\delb{\overline{\partial}}

\newcommand{\cA}{\mathcal{A}}
\newcommand{\cB}{\mathcal{B}}

\newcommand{\cD}{\mathcal{D}}

\newcommand{\cJ}{\mathcal{J}}

\newcommand{\cK}{\mathcal{K}}

\newcommand{\cP}{\mathcal{P}}
\newcommand{\cF}{\mathcal{F}}

\newcommand{\cW}{\mathcal{W}}
\newcommand{\XX}{\mathcal{X}}
\newcommand{\cY}{\mathcal{Y}}

\newcommand{\cG}{\mathcal{G}}
\newcommand{\cL}{\mathcal{L}}
\newcommand{\cO}{\mathcal{O}}

\newcommand{\cS}{\mathcal{S}}

\newcommand{\ccG}{\mathcal{G}^{\C}}
\newcommand{\Lie}{\operatorname{Lie}}

\newcommand{\cH}{\mathcal{H}}
\newcommand{\re}{\operatorname{Re}}

\newcommand{\Aut}{\operatorname{Aut}}

\newcommand{\Cal}{\operatorname{Cal}}

\newcommand{\Fut}{\operatorname{Fut}}

\newcommand{\End}{\operatorname{End}}
\newcommand{\Id}{\operatorname{Id}}

\newcommand{\im}{\operatorname{Im}}
\newcommand{\Ord}{\operatorname{Ord}}
\newcommand{\Ric}{\operatorname{Ric}}

\newcommand{\skw}{\operatorname{skw}}
\newcommand{\tr}{\operatorname{tr}}

\newcommand{\U}{\operatorname{U}}

\newcommand{\Sp}{\operatorname{Sp}}
\newcommand{\LSp}{\operatorname{\mathfrak{sp}}}
\newcommand{\G}{\operatorname{G}}

\newcommand{\contr}{\,\lrcorner\,}

\newcommand{\ol}[1]{\overline{#1}}

\begin{document}

\title[deformations of cscS metrics]
{Deformations of constant scalar curvature Sasakian metrics and K-stability}
\author[C. Tipler]{Carl Tipler}
\author[C. van Coevering]{Craig van Coevering}
\address{D\'epartement de Math\'ematiques, Universit\'e de Bretagne Occidentale, 6, avenue Victor Le Gorgeu, 29238 Brest Cedex 3 France}
\address{Department of Mathematics 15-01, U.S.T.C., Anhui, Hefei 230026, P. R. China}
\email{carl.tipler@univ-brest.fr ; craigvan@ustc.edu.cn}

\begin{abstract}
Extending the work of G. Sz\'ekelyhidi and T. Br\"{o}nnle to Sasakian manifolds we prove that
a small deformation of the complex structure of the cone of a constant scalar curvature Sasakian manifold admits a constant scalar curvature structure if it is K-polystable.  This also implies that a small deformation of the complex structure of the cone of a
constant scalar curvature structure is K-semistable.  As applications we give examples of constant scalar curvature Sasakian manifolds which are deformations of toric examples, and we also show that if a 3-Sasakian manifold admits a non-trivial transversal complex deformation then it admits a non-trivial Sasaki-Einstein deformation.
\end{abstract}

\maketitle

\section{Introduction}\label{sec:intro}

The existence of a K\"{a}hler-Einstein metric on a compact K\"{a}hler manifold $X$ with $c_1(X)>0$ by a famous conjecture
of S.-T. Yau~\cite{Yau93} is thought to be equivalent to some geometric invariant theory (GIT) notion of stability of $X$.
G. Tian later introduced the notion of K-stability~\cite{Tian97} and showed that it is a necessary condition for the existence of a
K\"{a}hler-Einstein metric.  Then S. Donaldson extended the notion of K-stability to any K\"{a}hler manifold whose
K\"{a}hler class is $c_1(\mathbf{L})$ for an ample line bundle $\mathbf{L}$.
The Yau-Tian-Donaldson conjecture states that the existence of a constant scalar curvature
K\"{a}hler (cscK) metric in $c_1(\mathbf{L})$ should be equivalent to the K-stability, or K-polystablity when there are non-trivial
holomorphic vector fields, of the polarized variety $(X,\mathbf{L})$.  It is known from the work of S. Donaldson~\cite{Don05},
J. Stoppa~\cite{Sto09} and T. Mabuchi~\cite{Mab08} that the existence of a cscK metric implies K-stability.

Since a Sasakian manifold is essentially an odd dimensional analogue of a K\"{a}hler manifold, and furthermore the category
of polarized K\"{a}hler manifolds $(X,\mathbf{L})$ embeds in the former as the class of regular Sasakian manifold, it is natural
to make a similar conjecture for Sasakian manifolds.

Besides its similarity with K\"{a}hler geometry interest in Sasakian manifolds
has had two main motivations.  First, the work in Sasaki-Einstein manifolds in the last fifteen years has been very prolific in producing examples of positive scalar curvature Einstein manifolds.  See for example~\cite{BGMR98,BoyGalKol05,Kol05}, and the recent survey~\cite{Spa11} and its references, since the literature is too vast to list.
The second impetus has been from theoretical physics with AdS/CFT correspondence which, in the most interesting dimensions,
provides a duality between field theories and string theories on $\operatorname{AdS}_5 \times M$ where supersymmetry requires the
five dimensional Einstein manifold $M$ to have a real Killing spinor (see~\cite{Mal98,KleWit99,AFHS98,MorPle99}).

A polarization of a Sasakian manifold $M$ is given by a Reeb vector field $\xi$ and the cone $Y=(C(M),\xi)$ is an affine variety polarized by $\xi$.  T. Collins and G. Sz\'ekelyhidi~\cite{ColSze12}
defined the notion of a test configuration for a polarized affine variety $(Y,\xi)$ and were able to define K-polystability
by extending the Futaki invariant to singular polarized affine varieties.
They proved that existence of a constant scalar curvature Sasakian (cscS) metric implies K-semistability.
This generalizes earlier work of J. Ross and R. Thomas~\cite{RosTho11} which defined K-stability for polarized K\"{a}hler orbifolds
and proved that the existence of a cscK orbifold metric implies K-semistability.  Collins and
Sz\'ekelyhidi extended this result to irregular Sasakian manifolds.

There are already well known obstructions to the existence of a Sasaki-Einstein or cscS metric on a polarized Sasakian manifold.
There is the Futaki invariant~\cite{BoyGalSim08,FutOnoWan09} which is defined just as for K\"{a}hler manifolds, and there are
the volume minimization results of Martelli, Sparks and Yau~\cite{MarSpaYau08}.  Also there are the Lichnerowicz and
Bishop obstructions of~\cite{GMSY07}, which are only non-trivial for non-regular Sasakian manifolds.  But most of the research in Sasaki-Einstein manifolds has concentrated on proving existence of examples using sufficient conditions which are probably far from necessary.  However the above results make the following conjecture natural.

\begin{conjint}\label{conj:Sasak-Yau-conj}
Let $(M,\xi)$ be a Sasakian manifold polarized by the Reeb vector field $\xi$ with polarized affine cone $(Y,\xi)$.
Then there exists a cscS structure compatible with $\xi$ and the fixed complex cone $Y=C(M)$ if and only if
$(Y, \xi)$ is K-polystable.
\end{conjint}

This article is concerned with a local study of stability.  That is given a cscS manifold $(M,\xi,g)$ with its polarized
affine cone $(Y,\xi)$, we consider small deformations of the complex structure $(Y_t, \xi)_{t\in\mathcal{U}}$ preserving $\xi$.
We use the contact perspective of W. He~\cite{He11}, which is analogous to the study of K\"{a}hler structures by fixing
the symplectic form and varying the almost complex structure due to S. Donaldson~\cite{Don97}.  Thus we fix a contact
manifold $(M,\eta,\xi)$ and consider the space $\cK$ of $(1,1)$-tensors $\Phi$ inducing the \emph{transversal almost complex
structure} of a structure $(g_{\Phi},\eta,\xi,\Phi)$ which is Sasakian if $\Phi$ is integrable.  The action of the exact contactomorphism group $\cG$ is then Hamiltonian with moment map
\[ \mu: \cK \rightarrow C_b^{\infty}(M)_0, \]
$\mu(\Phi) = s-s_0$ the scalar curvature of $g_{\Phi}$ and $C_b^{\infty}(M)_0$ are smooth functions invariant under the Reeb flow
with zero average.

If one fixes a transversal complex structure $\ol{J}$ one
gets the space of compatible Sasakian structures $\cS(\xi, \bar J)$, which is analogous to the space of K\"{a}hler metrics
in a fixed K\"{a}hler class.   We are able to prove the following.
\begin{thmint}\label{thmint:main}
Let $(M,\eta,\xi,\Phi_0)$ be a cscS manifold and $(M,\eta,\xi,\Phi)$ a nearby Sasakian manifold with
transverse complex structure $\bar J$. Then if $(M,\eta,\xi,\Phi)$ is K-polystable, there is a constant scalar curvature Sasakian structure in the space $\cS(\xi, \bar J)$.
\end{thmint}
The proof of Theorem~\ref{thmint:main} uses a technique of G. Sz\'ekelyhidi~\cite{Sze10} of constructing a finite dimensional
slice to the action of the ``complexification'' $\ccG$ of the contactomorphism group, and reducing the problem to one
of finite dimensional geometric invariant theory.  T. Br\"{o}nnle~\cite{Bro11} also obtained similar results independently.

An application of this technique is in determining when a small deformation $\ol{J}$ of the transversal complex structure of
a cscS metric $(M,\eta,\xi,\Phi_0)$ admits a compatible cscS structure in $\cS(\xi, \bar J)$.  This is reduced to the
polystability of the corresponding orbit of $G^{\G}$, $G=\Aut(\eta,\xi,\Phi_0)$, on $H^1(\cB^\bullet)$ where $\cB^\bullet$ is
the appropriate deformation complex.  This problem has been considered by the second author using analytic methods
in~\cite{vanCo12b,vanCo13}.

A consequence of this technique is the following result extending the result of~\cite{ColSze12}.
\begin{thmint}
Let $(M,\eta,\xi,\Phi_0)$ be a cscS manifold.  Then a small deformation $(M,\eta,\xi,\Phi)$ fixing the Reeb vector field is
K-semistable.
\end{thmint}

In the final section we give some examples of toric cscS manifolds admitting polystable deformations and thus cscS deformations,
and also non-polystable deformations.  We also show that ``real'' deformations of a 3-Sasakian manifold are polystable.
This shows that any 3-Sasakian manifold with non-trivial deformations of the transversal complex structure admits a non-trivial
Sasaki-Einstein deformation.

\subsection{Acknowledgments}\label{subsec:acknow}

This work has been presented for the first time at the 18th International Symposium on Complex Geometry at Sugadaira.
The first author would like to thank the organizers for their invitation and hospitality.
The authors want to thank Professors Charles Boyer, Akito Futaki, Tam\'as K\'alm\'an, Toshiki Mabuchi
and Christina Toennesen-Friedman for their interest in this work.  We would also like to thank the reviewer for helpful
comments which led to a revision of Proposition~\ref{prop:map}.

\section{Background on Sasakian manifolds, canonical metrics and K-stability}\label{sec:background}

\subsection{Sasakian structures}\label{subsec:def}
First we recall the definition of a Sasakian manifold and basics of Sasakian geometry. We refer
the reader to the monograph~\cite{BoyGal08} for more details.

\begin{defn}\label{def:Sasak}
A Riemannian manifold $(M,g)$ is a Sasakian manifold if the metric cone
$(C(M),\ol{g})=(\R_{>0} \times M, dr^2 +r^2 g)$ is K\"{a}hler with respect to some complex structure $I$, where $r$ is the
usual coordinate on $\R_{>0}$.
\end{defn}

Thus the dimension $n$ of $M$ is odd and denoted $n=2m+1$, while $C(M)$ is a complex manifold with $\dim_{\C} C(M) =m+1$.

In the following we will characterize Sasakian manifolds as a special type of metric contact
structure. We will identify $M$ with the submanifold $\{1\}\times M\subset C(M)$.
Let $r\del_r$ be the Euler vector field on $C(M)$,
we define a vector field tangent to $M$ by $\xi =Ir\del_r$.
Then $r\del_r$ is real holomorphic, $\xi$ is Killing with respect to both $g$ and
$\ol{g}$, and furthermore the orbits of $\xi$ are geodesics on $(M,g)$.
Define $\eta =\frac{1}{r^2}\xi\contr\ol{g}$, then we have
\begin{equation}
\eta =-\frac{I^* dr}{r} =d^c \log r,
\end{equation}
where $d^c =\sqrt{-1}(\ol{\del} -\del)$.  If $\omega$ is the K\"{a}hler form of $\ol{g}$, i.e.
$\omega(X,Y) =\ol{g}(IX,Y)$, then $\mathcal{L}_{r\del_r} \omega =2\omega$ which implies that
\begin{equation}
\label{eq:Kaehler-pot1}
\omega =\frac{1}{2}d(r\del_r \contr\omega) =\frac{1}{2}d(r^2 \eta)=\frac{1}{4}dd^c(r^2).
\end{equation}
From (\ref{eq:Kaehler-pot1}) we have
\begin{equation}\label{eq:Kaehler-pot2}
\omega=rdr\wedge\eta +\frac{1}{2}r^2 d\eta.
\end{equation}

We will use the same notation to denote $\eta$ and $\xi$ restricted to $M$.  Then (\ref{eq:Kaehler-pot2}) implies that
$\eta$ is a contact form with Reeb vector field $\xi$, since $\eta(\xi)=1$ and $\mathcal{L}_{\xi} \eta =0$.  Then
$(M, \eta, \xi)$ is a contact manifold.

Let $D\subset TM$ be the contact distribution which is defined by
\begin{equation}
D_x =\ker\eta_x
\end{equation}
for $x\in M$. Furthermore, if we restrict the almost complex structure to $D$, $J:=I|_D$, then $(D,J)$ is a strictly pseudoconvex CR structure on $M$.  There is a splitting of the tangent bundle $TM$
\begin{equation}\label{eq:split}
TM =D\oplus L_{\xi},
\end{equation}
where $L_{\xi}$ is the trivial subbundle generated by $\xi$.  Define a tensor $\Phi\in\End(TM)$ by
$\Phi|_D =J$ and $\Phi(\xi) =0$.  Then we denote the Sasakian structure by $(g,\eta,\xi,\Phi)$.
Definition~\ref{def:Sasak} is the simplest definition of a Sasakian structure, but a Sasakian structure is also frequently
defined as a metric contact structure satisfying an additional \emph{normality} condition.
We give some details that are needed in this article but for more details see~\cite{BoyGal08}.

Assume for now that we merely have a contact manifold $(M, \eta, \xi)$, with Reeb vector field $\xi$.
\begin{defn}
A $(1, 1)$-tensor field $\Phi: TM\rightarrow TM$ on a contact manifold $(M,\eta,\xi)$
is called an \emph{almost contact-complex structure} if
\[ \Phi \xi=0,\ \Phi^2=-Id+\xi\otimes \eta. \]
An almost contact-complex structure is called \emph{K-contact} if in addition, $\cL_{\xi} \Phi=0$.
\end{defn}

An almost contact-complex structure $\Phi$ is compatible with a metric $g$ if
\[ g(\Phi X,\Phi Y) =g(X,Y) -\eta(X)\eta(Y), \quad\text{for }X,Y\in\Gamma(TM).\]
If $g$ is a compatible metric and
\[ g(\Phi X,Y) =\frac{1}{2}d\eta(X,Y), \quad\text{for }X,Y\in\Gamma(TM),\]
then $(g,\eta,\xi,\Phi)$ is a called a \emph{metric contact structure}.
If in addition the Reeb vector field $\xi$ is Killing, then it is called a \emph{K-contact metric structure}.

\begin{defn}\label{defn:con-comp-st}
An almost contact-complex structure $\Phi$ on a contact manifold $(M,\eta,\xi)$ is compatible with $\eta$ if
\[ d\eta(\Phi X, \Phi Y)=d\eta(X, Y),\text{ and }d\eta(X, \Phi X)>0 \text{ for } X\in\ker\eta,\ X\ne 0.\]
\end{defn}

If $\Phi$ is compatible with $\eta$, then one defines a Riemannian metric by
\[ g_{\Phi}(X, Y)=\frac{1}{2}d\eta(X, \Phi Y)+\eta(X)\eta(Y), \]
and $(\eta, \xi, \Phi, g_{\Phi})$ is a contact metric structure on $M$.  This metric structure is K-contact
if  $\cL_{\xi} \Phi=0$, that is $\Phi$ is K-contact.  Henceforth, we will only consider compatible almost contact-complex
structures.

If $(\eta, \xi, \Phi, g_{\Phi})$ is K-contact, because
\[ d\eta(X,Y) =2g_{\Phi}(\Phi X,Y),\]
we have $\Phi =\nabla\xi$, where $\nabla$ is the Levi-Civita connection of $g_{\Phi}$.

If $(\eta,\xi,\Phi,g)$ is a contact metric structure, then we define $\Phi$ to be \emph{normal} if
\begin{equation}\label{eq:Phi-normal}
N_{\Phi}(X,Y) =d\eta(X,Y)\otimes\xi\quad\text{for all }X,Y\in\Gamma(TM),
\end{equation}
where $N_\Phi$ is the Nijenhuis tensor
\[ N_{\Phi} (X,Y) :=-\Phi^2 [X,Y] + \Phi([\Phi X,Y] +[X,\Phi Y]) -[\Phi X,\Phi Y],\text{ for all }X,Y\in\Gamma(TM).\]
A Sasakian structure $(\eta,\xi,\Phi,g)$ can be defined as a normal contact metric structure.
Note that (\ref{eq:Phi-normal}) implies that $\cL_\xi \Phi =0$, so it is K-contact.

One can alternatively define a Sasakian structure to be a K-contact metric structure $(\eta, \xi, \Phi, g)$ for which the
CR structure $(D,J)$, where $D_x :=\ker\eta_x$ and $J:=\Phi|_D$, is integrable, that is
\begin{equation}\label{eq:CR-int}
N_{(D,J)} \equiv 0,
\end{equation}
It turns out (see \cite{BoyGal08}) that (\ref{eq:Phi-normal}) is equivalent
to $\cL_\xi \Phi =0$ and (\ref{eq:CR-int}) for a metric contact structure $(\eta, \xi, \Phi, g)$.

We define an almost K-contact-complex structure $\Phi$ to be \emph{integrable} if (\ref{eq:CR-int}) holds.
If $\Phi$ is a K-contact almost contact-complex structure compatible with $\eta$, then $(\eta, \xi, \Phi, g_{\Phi})$ is Sasakian
if and only if $\Phi$ is integrable.

One can also define a Sasakian structure, compatible with a Riemannian metric g, to be a unit length Killing field
$\xi$ such that the tensor $\Phi X =\nabla_X \xi$ satisfies the condition
\[ \nabla_X \Phi(Y)=g(\xi, Y)X-g(X, Y)\xi. \]
See~\cite{BoyGal08} for details.

We will frequently denote a Sasakian structure by $(g,\eta,\xi,\Phi)$ even though specifying
$(g,\xi),\ (g,\eta),$ or $(\eta, \Phi)$ is enough to determine the Sasakian structure.

\subsection{Transverse K\"{a}hler structure}\label{subsec:trans-Kah}

The \emph{Reeb foliation} $\mathscr{F}_\xi$ on $M$ generated by the action of $\xi$ will be important in the sequel.
Note that it has geodesic leaves but in general the leaves are not compact.  If the leaves are compact, or equivalently
$\xi$ generates an $S^1$-action, then $(g,\eta,\xi,\Phi)$ is said to be a \emph{quasi-regular} Sasakian structure, otherwise it is
irregular.  If this $S^1$ action is free, then $(g,\eta,\xi,\Phi)$ is said to be \emph{regular}.  In this last case
$M$ is an $S^1$-bundle over a manifold $X$, which we will see below is K\"{a}hler.  If the structure is merely quasi-regular, then
the leaf space has the structure of a K\"{a}hler orbifold.
In general, in the irregular case, the leaf space is not even Hausdorff but we will make use of the transversally K\"{a}hler
property of the foliation $\mathscr{F}_\xi$ which we discuss next.

The vector field $\xi -\sqrt{-1}I\xi =\xi +\sqrt{-1}r\del_r$ is holomorphic on $C(M)$.  If we denote by $\tilde{\C}^*$ the
universal cover of $\C^*$, then $\xi +\sqrt{-1}r\del_r$ induces a holomorphic action
of $\tilde{\C}^*$ on $C(M)$.  The orbits of $\tilde{\C}^*$ intersect $M\subset C(M)$ in the orbits of the Reeb foliation
generated by $\xi$.  We denote the Reeb foliation by $\mathscr{F}_\xi$.  This gives $\mathscr{F}_\xi$ a transversely holomorphic
structure.

The foliation $\mathscr{F}_{\xi}$ together with its transverse holomorphic structure is given by an open covering
$\{U_\alpha \}_{\alpha\in A}$ and submersions
\begin{equation}\label{eq:fol-chart}
\Pi_\alpha :U_\alpha \rightarrow W_\alpha \subset\C^{m}
\end{equation}
such that
when $U_\alpha \cap U_\beta \neq\emptyset$ the map
\[\Phi_{\beta\alpha} =\Pi_\beta \circ\Pi_\alpha^{-1} :\Pi_{\alpha}(U_\alpha \cap U_\beta) \rightarrow\Pi_{\beta}(U_\alpha \cap U_\beta) \]
is a biholomorphism.

Note that on $U_\alpha$ the differential $d\Pi_\alpha :D_x \rightarrow T_{\Pi_\alpha(x)}W_\alpha$ at $x\in U_\alpha$ is
an isomorphism taking the almost complex structure $J_x$ to that on $T_{\Pi_\alpha(x)}W_\alpha$.
Since $\xi\contr d\eta =0$ the 2-form $\frac{1}{2}d\eta$ descends to a form $\omega_\alpha^T$ on $W_\alpha$.  Similarly,
$g^T =\frac{1}{2}d\eta(\cdot,\Phi\cdot)$ satisfies $\mathcal{L}_\xi g^T =0$ and vanishes on vectors tangent to the leaves, so
it descends to an Hermitian metric $g^T_\alpha$ on $W_\alpha$ with K\"{a}hler form $\omega_\alpha^T$.  The K\"{a}hler metrics
$\{g_\alpha ^T \}$ and K\"{a}hler forms $\{\omega_\alpha^T \}$ on $\{ W_\alpha\}$ by construction are isomorphic on the overlaps
\[ \Phi_{\beta\alpha} : \Pi_{\alpha}(U_\alpha \cap U_\beta) \rightarrow\Pi_{\beta}(U_\alpha \cap U_\beta).\]
We will use $g^T$, respectively $\omega^T$, to denote both the K\"{a}hler metric, respectively K\"{a}hler form, on the
local charts and the globally defined pull-back on $M$.

If we define $\nu(\mathscr{F}_\xi) =TM/{L_\xi}$ to be the normal bundle to the leaves, then we can generalize the above concept.

\begin{defn}
A tensor $\Psi\in\Gamma\bigl((\nu(\mathscr{F}_\xi)^*)^{\otimes p} \bigotimes\nu(\mathscr{F}_\xi)^{\otimes q}\bigr)$ is \emph{basic}
if $\mathcal{L}_V \Psi =0$ for any vector field $V\in\Gamma(L_\xi)$.
\end{defn}
  Note that it is sufficient to check the above property for $V=\xi$.
Then $g^T$ and $\omega^T$ are such tensors on $\nu(\mathscr{F}_\xi)$.  We will also make use of the bundle isomorphism
$\Pi:D \rightarrow\nu(\mathscr{F}_\xi)$, which induces an almost complex structure $\ol{J}$ on $\nu(\mathscr{F}_\xi)$ so that
$(D,J)\cong(\nu(\mathscr{F}_\xi),\ol{J})$ as complex vector bundles.  Clearly, $\ol{J}$ is basic and is mapped to the
natural almost complex structure on $W_\alpha$ by the local chart $d\Pi_\alpha :D_x \rightarrow T_{\Pi_\alpha(x)}W_\alpha$.

To work on the K\"{a}hler leaf space we define the Levi-Civita connection of $g^T$ by
\begin{equation}
\nabla^T_X Y =\begin{cases}
\Pi_\xi(\nabla_X Y) & \text{ if }X, Y\text{ are smooth sections of }D, \\
\Pi_\xi([V,Y]) & \text{ if } X=V\text{ is a smooth section of }L_\xi,
\end{cases}
\end{equation}
where $\Pi_\xi :TM \rightarrow D$ is the orthogonal projection onto $D$.  Then $\nabla^T$ is the unique torsion free connection
on $D\cong\nu(\mathscr{F}_\xi)$ so that $\nabla^T g^T=0$.  Then for $X,Y\in\Gamma(TM)$ and $Z\in\Gamma(D)$ we have the
curvature of the transverse K\"{a}hler structure
\begin{equation}
R^T(X,Y)Z =\nabla^T_X \nabla^T_Y Z -\nabla^T_Y \nabla^T_X Z -\nabla^T_{[X,Y]} Z,
\end{equation}
and similarly we have the transverse Ricci curvature $\Ric^T$ and scalar curvature $s^T$.  We will denote the
transverse Ricci form by $\rho^T$.
From O'Neill's tensors computation for Riemannian submersions \cite{ONe66} and elementary properties of Sasakian structures
we have the following.
\begin{prop}\label{prop:Sasaki-Ric}
Let $(M,g,\eta,\xi,\Phi)$ be a K-contact manifold of dimension $n=2m+1$, then
\begin{thmlist}
\item  $\Ric_g (X,\xi) =2m\eta(X),\quad\text{for }X\in\Gamma(TM)$,\label{eq:submer-Ric-Reeb}
\item  $\Ric^T (X,Y) =\Ric_g (X,Y) +2g^T(X,Y),\quad\text{for }X,Y\in\Gamma(D),$
\item  $s^T =s +2m.$\label{eq:submer-scal}
\end{thmlist}
\end{prop}

\begin{defn}
A constant scalar curvature Sasakian (cscS) manifold $(M,g,\eta,\xi,\Phi)$ is a Sasakian manifold with
$s^T$ constant, or equivalently $s_g$ constant.
\end{defn}

It will be convenient at times to consider the larger class of K-contact structures $(g,\eta,\xi,\Phi)$
compatible with $(M,\eta,\xi)$.  In this case moment map of~\cite{Don97,He11} is the scalar curvature of
the \emph{Chern connection} $\nabla^c$ on $(\nu(\mathscr{F}_\xi),\ol{J})$
\[ \nabla^c_X Y =\nabla^T_X Y -\frac{1}{2}\ol{J}\nabla^T_X \ol{J}(Y).  \]
So in considering K-contact structures we will consider a different $s_c^T$ than in Proposition~\ref{prop:Sasaki-Ric}.

Let $\mathcal{S}(\xi)$ be the space of Sasakian structures $(\tilde{g},\tilde{\eta},\tilde{\xi},\tilde{\Phi})$ on $M$ with
$\tilde{\xi}=\xi$.  For any $(\tilde{g},\tilde{\eta},\tilde{\xi},\tilde{\Phi})\in\mathcal{S}(\xi)$ the 1-form
$\beta=\tilde{\eta}-\eta$ is basic, so $[d\tilde{\eta}]_b=[d\eta]_b$, where $[\,\cdot\,]_b$ denotes the basic cohomology
class of a basic closed form.  Thus $[\omega^T]_b \in H^2_b(M/\mathscr{F}_\xi,\R)$ (see~\cite{BoyGal08} for more on basic cohomology) is the same for every Sasakian structure in $\mathcal{S}(\xi)$.
Thus, as first observed in~\cite{BoyGalSim08}, fixing the Reeb vector field is the closest analogue to a polarization in K\"{a}hler geometry, and we say that the Reeb vector field $\xi$ \emph{polarizes} the Sasakian manifold.

We will consider the space of Sasakian structures $\cS(\xi,\ol{J})$ with fixed Reeb vector field and fixed transversal complex structure $\ol{J}$.  We define
\begin{equation}\label{eq:trans-def}
\cH_{\Phi}=\{\phi\in C^\infty_b(M):\ \eta_\phi\wedge (d\eta_\phi)^n\neq 0 \}
\end{equation}
where for any $\phi\in \cH$, we define a new Sasakian structure $(\eta_\phi, \xi, \Phi_\phi, g_\phi)$ with the same Reeb vector field $\xi$ such that
\begin{equation}\label{eq:trans-def2}
\eta_\phi =\eta+d^c_b\phi,\ \Phi_\phi=\Phi-\xi\otimes d^c_b\phi \circ \Phi,
\end{equation}
the transversal K\"{a}hler form is $\omega^T_{\phi} =\frac{1}{2}d\eta_\phi =\frac{1}{2}d\eta+\frac{1}{2}d_b d^c_b \phi$, and $g_\phi$ is as in Definition~\ref{defn:con-comp-st}.

Note that $D$ and $\Phi_\phi$ vary but $(\eta_\phi, \xi, \Phi_\phi, g_\phi)$ has the same transverse holomorphic structure
and same complex structure on $C(M)$ as $(\eta, \xi, \Phi, g)$ (Prop. 4.1 in \cite{FutOnoWan09}).
On the other hand,  if $(\tilde \eta, \xi, \tilde \Phi, \tilde g)\in\cS(\xi,\ol{J})$ is another Sasakian structure  with the same Reeb vector field $\xi$ and the same transverse complex structure, then there exists unique functions
$\phi\in \cH,\ \psi\in C^\infty_b(M)$ up to addition of
a constant and $\alpha\in H^1_b$ a harmonic 1-form such that
\begin{equation}\label{eq:cont-decom}
\tilde\eta =\eta +\alpha + d^c_b \phi +d_b\psi,
\end{equation}
See~\cite[Lemma 3.1]{BoyGalSim08}.
Note that $\psi$ is given by a gauge transformation $\exp(\psi\xi)$ of $M$.  Since $\alpha$ and $d_b\psi$ do not effect the transversal K\"{a}hler structure they will not be important.  Thus $\cS(\xi, \bar J)$ can be viewed as the analogue of the set of K\"ahler metrics
in a fixed K\"ahler class.

Boyer-Galicki-Simanca \cite{BoyGalSim08} proposed to seek the extremal Sasakian metrics  to represent $\cS(\xi, \bar J)$, by extending Calabi's extremal problem to Sasakian geometry.  We denote by $\mathfrak{M}(\xi,\ol{J})$ the metrics associated with Sasakian structures in $\mathcal{S}(\xi,\ol{J})$.  We define the Calabi functional by
\begin{equation}\label{eq:Calabi-Sasak}
\begin{array}{rcl}
\mathfrak{M}(\xi,\ol{J}) & \overset{\Cal}{\longrightarrow} & \R \\
g & \mapsto & \int_M (s-s_0)^2 \, d\mu_g,
\end{array}
\end{equation}

where $s_0$, the average of $s$ is independent of the structure in $\cS(\xi, \bar J)$.
By Proposition~\ref{prop:Sasaki-Ric}.\ref{eq:submer-scal} $s_0 =s_0^T -2m$ where
\[\begin{split}
s^T_0 & =\frac{\int_M s^T\, d\mu}{\int_M d\mu} \\
      & =\frac{\int_M 4m\pi c_1(\mathscr{F}_\xi)\wedge\eta \wedge\bigl(\omega^T \bigr)^{m-1}}{\int_M \eta\wedge\bigl(\omega^T \bigr)^{m}}. \\
\end{split}\]

A \emph{Sasaki-extremal} metric $g\in\mathfrak{M}(\xi,\ol{J})$ is a critical point of $\Cal$.  As in the K\"{a}hler
case, the Euler-Lagrange equations of (\ref{eq:Calabi-Sasak}) show that this is equivalent to
the basic vector field $\del_g^{\#} s_g :=(\delb s_g)^{\#}$ being transversely holomorphic.
Thus constant scalar curvature Sasakian metrics are examples, and furthermore a Sasaki-extremal metric is of constant
scalar curvature precisely when the transversal Futaki invariant is zero (cf.~\cite{BoyGalSim08}).  The results of this
article can be extended to Sasaki-extremal metrics using relative K-stability.

\subsection{Moment map interpretation}\label{sec:mmap}

We will be interested in finding constant scalar curvature Sasakian metrics.  Although, much of what follows can be applied
more generally to Sasaki-extremal metrics by considering that as a relative version of the cscS case.
In~\cite{He11} W. He gave an interpretation of this problem in term of a moment map, as Donaldson did for the cscK case~\cite{Don97}.

Let $\cG$ be the group of \emph{strict contactomorphisms}.  This is the group of diffeomorphisms $f: M\rightarrow M$ which satisfy $f^{*}\eta=\eta$.  It has Lie algebra
\[ \Lie(\cG)=\{X\in \Gamma(TM),\ \cL_X\eta=0\},\]
the space of {\it strict contact} vector fields.
The space of basic functions $C_b^\infty(M)$ is isomorphic to $\Lie(\cG)$.
For any $X\in\Lie(\cG)$ define $H_X =\eta(X)$, and conversely
for each basic function $H\in C_b^\infty(M)$ , there exists a unique strict contact vector field $X=X_H \in\Gamma(TM)$ which satisfies
\[ H=\eta(X),\ X\contr d\eta=-dH. \]
The Poisson bracket is then defined by
\[ \{F, H\}=\eta([X_F, X_H]),\]
and $H\mapsto X_H$ is a Lie algebra isomorphism.

We will use the natural $\cG$-invariant $L^2$ inner product on $C_b^\infty(M)$
\begin{equation}\label{eq:L2}
\langle f, h \rangle=\int_M f h\, d\mu,
\end{equation}
where $d\mu=(2^m m!)^{-1}\eta\wedge (d\eta)^m$ is a volume form determined by $\eta$.

The group of strict contactomorphisms $\cG$ acts on the space $\cK$ of K-contact structures which are compatible
with $\eta$ via
\[ (f, \Phi)\rightarrow f_{*}\Phi f^{-1}_{*}.\]

Moreover, $\cK$ can be endowed with a K\"ahler structure~\cite{He11} for which it is an infinite dimensional symmetric space and
for which $\cG$ acts by biholomorphisms and isometries.  First note that
\[ T_{\Phi}\cK=\lbrace  A\in \End(TM)\ :\ A\xi=0,\ \cL_{\xi}A=0,\ A\Phi +\Phi A=0,\ d\eta(AX,Y)+d\eta (X,AY)=0,\forall X,Y\in\Gamma(TM)\  \rbrace.\]
An almost-complex structure $\cJ$ is defined on $\cK$ by
\[ \cJ A= \Phi A. \]
To each $\Phi \in \cK$, with $\eta$ we can associate a Sasakian metric $g_{\Phi}$ which induces a metric on tensors. We
define on $\cK$ a \emph{weak} Riemannian metric
\[ \begin{split}
g_\cK( A, B) & = \int_M \langle A, B\rangle_{g_{ \Phi}}\, d\mu_{\eta} \\
                         & =\int_M \tr (AB)\, d\mu_{\eta}
\end{split}\]
which is Hermitian with respect to $\cJ$

Let $\Phi_0$ be a fixed K-contact structure.  Let $\End_{S,\Phi_0}(D)$ be the basic endomorphisms of $D=\ker\eta$ symmetric
with respect to $g_0^T$ and anti-commuting with $\Phi_0$.  For convenience we identify an endomorphism of $D$ with an
endomorphism of $TM$ by acting by zero on the second factor of (\ref{eq:split}).
Define
\begin{equation}\label{eq:chart-dom}
K_{\Phi_0} :=\{ Q\in\End_{S,\Phi_0}(D)\ :\  Id -Q^2 >0\; \}.
\end{equation}
We have a chart
\begin{equation}\label{eq:chart}
\begin{gathered}
\Psi_{\Phi_0} :K_{\Phi_0}\rightarrow\cK \\
Q \mapsto \Phi_0 (Id +Q)(Id -Q)^{-1}
\end{gathered}
\end{equation}
One can show that $\Psi_{\Phi_0}$ is a bijection.  Furthermore, one can easily compute the differential $d\Psi_{\Phi_0}$
of $\Psi_{\Phi_0}$ at $Q$
\begin{gather*}
d\Psi_{\Phi_0} : \End_{S,\Phi_0}(D)\rightarrow T_{\Phi}\cK \\
A \mapsto 2\Phi_0 (Id -Q)^{-1} A(Id -Q)^{-1},
\end{gather*}
and check that $d\Psi_{\Phi_0} \circ J =\cJ\circ d\Psi_{\Phi_0}$, where $J$ is the complex structure $A\mapsto\Phi_0\circ A$ on
$K_{\Phi_0}$.  Therefore the maps (\ref{eq:chart}) are holomorphic charts.  Also, arguments as in the symplectic case
show that the 2-form
\[ \Omega_{\cK}(A,B) =\int_M \tr (\Phi AB)\, d\mu_{\eta},\]
is closed.  See~\cite{FujSch88} and~\cite{Smo07} for more details.

Define $K^s_{\Phi_0},\ s>n+1,$ as in
(\ref{eq:chart-dom}) but with sections in Sobolev space $L^{2,s} (\End_{S,\Phi_0}(D))$, and consider the charts (\ref{eq:chart})
on $K^s_{\Phi_0}$.  The above arguments show that the space of $L^{2,s}$ K-contact structures $\cK^s$ is a smooth complex Hilbert
manifold.  And $\cK$ has the structure of a smooth complex ILH-manifold.

An almost contact-complex structure $\Phi_0$ can also be identified with a splitting
\[ D\otimes\C \cong\nu(\mathscr{F}_\xi)\otimes\C =T^{1,0}(\Phi_0)\oplus T^{0,1}(\Phi_0),\]
into $\sqrt{-1}$ and $-\sqrt{-1}$ eigenspaces of $\Phi_0$.

Suppose $\Phi_0$ is a K-contact complex structure.  If $\Phi$ is another K-contact complex structure then
$\Phi =\Phi_0 (Id+Q)(Id-Q)^{-1}$, with $Q\in K_{\Phi_0}$.  If we extend $Q$ to
\[Q:\nu(\mathscr{F}_\xi)\otimes\C \rightarrow \nu(\mathscr{F}_\xi)\otimes\C,\]
then
\[ Q =\begin{pmatrix} 0 & \ol{P} \\ P & 0 \end{pmatrix},\]
where $P: T^{1,0}(\Phi_0)\rightarrow T^{0,1}(\Phi_0)$.

This gives a useful complex parametrization of $\cK$.
\begin{prop}\label{prop:com-chart}
Given a K-contact complex structure $\Phi_0$, the manifold $\cK$ is parameterized by operators
$P:T^{1,0}(\Phi_0)\rightarrow T^{0,1}(\Phi_0)$ satisfying the following:
\begin{thmlist}
\item After lowering an index $P^\flat \in\Gamma(S^2(\Lambda^{1,0}_b))$, basic symmetric tensors, and
\item $Id-\ol{P}P>0$.
\end{thmlist}
And one has
\[ T^{1,0}(\Phi) =\im(Id -P),\quad T^{0,1}(\Phi) =\im(Id-\ol{P}),\]
where $\Phi =\Phi_0 (Id+Q)(Id-Q)^{-1},\ Q=\frac{1}{2}(P+\ol{P})$.
\end{prop}

The subspace $\cK^i \subseteq\cK$ of Sasakian structures is the subvariety for which (\ref{eq:CR-int}) is satisfied.
In the complex parametrization this can be written
\[ N(P) =\delb_b P +[P,P] =0.\]

The main result of \cite{He11} extends the work of~\cite{Don97} to give the following.
\begin{thm}
The map $\mu:\cK\rightarrow C_b^\infty(M)_0$ with $\mu(\Phi)=s^T(\Phi)-s^T_0$ is an \emph{equivariant} moment map for the $\cG$-action
on $\cK$, where $C_b^\infty(M)_0$ is identified with its dual under the pairing (\ref{eq:L2}).

Here $s^T$ denotes the Hermitian scalar curvature when $\Phi$ is not in $\cK^i$.
\end{thm}

The Lichnerowicz operator
\begin{equation}\label{eq:Lich}
\cP_{\Phi} : C^\infty_b(M)\rightarrow T_{\Phi}\cK,
\end{equation}
$\cP_{\Phi} (H) =\cL_{X_H}\Phi$ gives the infinitesimal action of $\cG$ on $\cK$.

The theorem reads
\begin{equation}\label{eq:mmap}
\Omega_{\cK}(\cP_{\Phi}(H),A)=\langle H, D s^T(A)\rangle,\quad\forall A \in T_{\Phi}\cK.
\end{equation}

If we were in the finite dimensional situation, then the Kempf-Ness theorem would lead to the identification
\[ \cK^{s}{/\!/}\ccG=\mu^{-1}(0)/\cG, \]
where $\cK^{s}$ are the polystable points in $\cK$ and $\ccG$ is the complexified group.
A constant transverse scalar curvature metric, which is a zero point of the moment map $\mu$,
would correspond to a polystable complex orbit of the $\ccG$ action.

But there are two problems in
this situation.  First, in this infinite dimensional situation there is no local compactness allowing the usual arguments.
And secondly, the complexification $\ccG$ of $\cG$ does not exist as a group.

Although $\ccG$ does not exist we can define the action of the complexified Lie algebra $C^\infty_b(M,\C)$
of $\cG$ on $\cK$ since it is a complex manifold.  We extend (\ref{eq:Lich}) to
\begin{equation}\label{eq:Lich-com}
\cP_{\Phi} : C^\infty_b(M,\C)\rightarrow T_{\Phi}\cK,
\end{equation}
by taking $\i H$, $H\in C^\infty_b(M)$ to $\Phi\cL_{X_H}\Phi$.
Then we say that a smooth path $\Phi(t)\in\cK$ lies in an orbit of $\ccG$ if
\begin{equation}\label{eq:GC-orbit}
 \dot{\Phi}(t)\in\im\cP_{\Phi(t)},\quad\forall t.
\end{equation}

Note that the integrability of $\Phi\in\cK^i$ does not imply that $\cL_{\Phi X_\phi} \Phi$ equals $\Phi\cL_{X_\phi}\Phi$.
In fact, an easy computation using (\ref{eq:Phi-normal}) shows that
\begin{equation}\label{eq:comp-act}
 \cL_{\Phi X_\phi}\Phi(Y) =\Phi\cL_{X_\phi}\Phi(Y) +d\phi(Y)\otimes\xi,\text{ for }Y\in TM.
\end{equation}
But this shows one does have equality on the level of the transverse K\"{a}hler structure $(\omega^T,\ol{J})$, since
the last term acts trivially on basic forms.

\begin{prop}\label{prop:GC-orbit}
Let $(M,\eta,\xi,\Phi,g)$ be Sasakian.  Then up to a diffeomorphism preserving the Reeb foliation $\mathscr{F}_\xi$
the $\ccG$ orbit of $\Phi$ consists precisely of all structures $(\eta_\phi,\xi,\Phi_\phi ,g_\phi)$ with $\phi\in C^\infty_b(M)$ as in (\ref{eq:trans-def2}).
\end{prop}
\begin{proof}
Let $\phi\in\cH_{\Phi}$.  Define the Sasakian structure $(\eta_t,\xi,\Phi_t),\ 0\leq t\leq 1,$
\begin{equation}\label{eq:cont-def}
 \eta_t =\eta +td^c \phi,\quad \Phi_t =\Phi-t\xi\otimes d\phi.
\end{equation}
Let $X^t_\phi$ be the strict contact vector field for $\eta_t$ with Hamiltonian $\phi$, and define
$V_t =\Phi_t X^t_\phi$.  We have
\[ \cL_{V_t} \eta_t = V_t \contr d\eta_t =-d^c \phi.\]
If $f_t$ is the flow associated to $V_t$, then
\[ \frac{d}{dt} f_t^* \eta_t =f_t^* \cL_{V_t} \eta_t +f_t^* d^c \phi =0. \]
Therefore $f_t^* \eta_t =\eta$, and the structure $(\eta_t,\Phi_t)$ is isometric to $(\eta,f_t^*\Phi_t)$,
where $f_t^*\Phi_t =f^{-1}_{t*}\circ\Phi_t \circ f_{t*}$.

We claim that $(\eta,f_t^*\Phi_t)$ is in the orbit of $\ccG$.  In fact
\begin{equation}
\begin{split}
\frac{d}{dt}f_t^*\Phi_t & = f_t^* \cL_{V_t}\Phi_t +f_t^*\bigl(-\xi\otimes d\phi \bigr) \\
                        & = f_t^*\bigl(\Phi_t \cL_{X^t_\phi}\Phi_t \bigr) \\
                        & = \bigl(f_t^* \Phi_t \bigr)\cL_{X_{f_t^* \phi}}\bigl(f_t^* \Phi_t \bigr),
\end{split}
\end{equation}
where $X_{f_t^* \phi}$ is the contact vector field for $\eta$ with Hamiltonian $f_t^* \phi$.
Thus $f_t^*\Phi_t$ satisfies (\ref{eq:GC-orbit}).

Conversely, suppose that $\Phi_t$ is a smooth path of almost contact-complex structures so that
$(\eta,\xi,\Phi_t)$ is Sasakian and $\Phi_0 =\Phi$.  Suppose this is contained in the $\ccG$
orbit of $\Phi$.  After possibly acting by contactomorphisms we may assume
\begin{equation}
\frac{d}{dt} \Phi_t =\Phi_t \cL_{X_{\phi_t}} \Phi_t,
\end{equation}
where $\phi_t \in C^\infty_b (M)$ is a smooth path and $X_{\phi_t}$ is the associated contact vector field.
Let $f_t$ be the flow of $\Phi_t X_{\phi_t}$.
Then
\[ \begin{split}
\frac{d}{dt}f_t^* \eta & = f_t^*\bigl(-d_t^c \phi_t \bigr) \\
                    & = -d^c \bigl( f_t^*\phi_t \bigr),
\end{split}\]
where $d^c_t$ is with respect to the transversal complex structure induced by $\Phi_t$.
The second equality is because $f_t^*\Phi_t$ and $\Phi$ induce the same transversal complex structure
from (\ref{eq:comp-act}).  And
\begin{equation}
f_t^* \eta -\eta =-d^c \int_0^t f^*_s \phi_s \,ds.
\end{equation}
Define $\eta_t =f_t^*\eta$, then $(\eta,\Phi_t)$ is isometric to $(\eta_t, f_t^*\Phi_t)$.
We have $\eta_t =\eta +d^c H_t$ where $H_t =-\int_0^t f^*_s \phi_s \,ds$.  But since
$f_t^*\Phi_t$ and $\Phi$ induce the same transversal complex structure, we must have
$f_t^*\Phi_t =\Phi - \xi\otimes dH_t$.
\end{proof}

\begin{rmk}
This orbit does not effect the harmonic $H_b^1$ or exact components in (\ref{eq:cont-decom}).
The exact component in (\ref{eq:cont-decom}) is given by the diffeomorphism $\exp(\psi\xi)$
Thus $\cS(\xi,\ol{J})$ consists of the $\ccG$ orbit of $\phi$ and variations by $H_b^1$.
But since these variations do not affect the transversal K\"{a}hler structure, this wont cause any issues.
\end{rmk}

\subsection{K-stability for Sasakian manifolds}\label{subsec:K-stab}

T. Collins and G. Sz{\'e}kelyhidi~\cite{ColSze12} defined test configurations and K-polystabiliy for irregular Sasakian manifolds
extending the work of J. Ross and R. Thomas~\cite{RosTho11} on the quasi-regular case.
To define test configurations for Sasakian manifolds, Collins and Sz\'ekelyhidi gave an algebraic interpretation of Reeb vector fields
on Sasakian manifolds which we recall now. Let $(M,g,\eta,\xi, \Phi)$ be a Sasakian manifold, then the metric cone
$(C(M)=\R_+\times M,\ \overline{g}=dr^2+r^2g)$ over $(M,g)$ is K\"ahler for the complex structure $I$.
As explained in \cite[Section 2]{ColSze12}, the cone $Y=C(M)\cup\{0\}$ is an affine variety with isolated singularity at $0$.

We may consider the cone $(Y,\xi)$ polarized by the Reeb vector field representing a polarized Sasakian manifold.
This is because any two Sasakian structures $(g, \eta, \xi,\Phi)$ and $(\tilde{g}, \tilde{\eta},\xi,\tilde{\Phi})$
with the same polarized cone $(Y,\xi)$ differ as in (\ref{eq:cont-decom}) with $\alpha =0$.
The Reeb vector field generates a torus action $T\subset \Aut(Y)$ with $\xi\in\t =\Lie(T)$, which of course extends to an
algebraic torus action $\TT\subset\Aut(Y)$.
Let $\cO_Y$ be the structure sheaf of $Y$ and consider the weight decomposition
\[H^0(Y,\cO_Y)=\sum_{\alpha\in\cW_T} H^0(Y,\cO_Y)_{\alpha} \]
where $\cW_T \cong\Z^k,\ k=\dim_{\C} \TT,$ are the weights of the $\TT$-action.

Then from $\eta(\xi)>0$ and \cite[proposition 2.1.]{ColSze12}, $\alpha(\xi)>0$ for each weight $\alpha\in\cW_T \setminus\{0\}$.
It turns out that the fact that the Reeb vector field acts with positive weights on the non-constant
functions of $Y$ gives an algebraic characterization of the Reeb cone:
\[ \lbrace \xi'\in \t\ :\ \eta(\xi')>0 \rbrace = \lbrace \xi'\in \t\ :\ -\sqrt{-1}\alpha(\xi')>0, \forall \alpha\in\cW_T \setminus\{0\}\rbrace.\]
It suggests the following algebraic definition of a Reeb field:
\begin{defn}
 A Reeb field on an affine scheme $Y$ with torus $T\subset \Aut{Y}$ is an element $\xi'\in\t$ such that
\[ -\sqrt{-1}\alpha(\xi')>0\; \forall \alpha\in\cW_T \setminus\{0\}. \]
\end{defn}

On the polarized cone $(Y,\xi)$ it remains to define the notion of a compatible K\"{a}hler metric.
\begin{defn}
A K\"ahler metric on an affine scheme $Y$ is compatible with a Reeb field $\xi\in \t$
if there exists a $\xi$-invariant function $r : Y \rightarrow \R_+$ such that $\om=\frac{\i}{2} \del \delb r^2$ and
$\xi=I(r\del_r)$ where $I$ is the almost complex structure on $Y$.
\end{defn}

Any polarized affine variety $(Y,\xi)$ smooth except at possibly one point admits a K\"{a}hler metric $\om$ compatible with
the Reeb field $\xi$, and $(Y,\xi,\om)$ is the metric cone over a Sasakian manifold.
To see this let $T$ be the torus generated by $\xi$ and choose sufficiently many $T$-homogenous generators
$\{f_1,\ldots,f_d \},\ f_i \in H^0(Y,\cO_Y)_{\alpha_i},\ \alpha_i \in\cW_T$, then

\[(f_1,\ldots, f_d): Y\rightarrow\C^d \]

is an embedding with $T$ acting diagonally on $\C^d$.  Then there exists a Sasakian structure on the sphere
$S^{2d-1}$ with Reeb vector field $\hat{\xi}$ restricting to $\xi$.
In particular, a Sasakian manifold $(M,g,\eta,\xi,\Phi)$ can equivalently be defined as an algebraic scheme $Y=C(M)\cup\{0\}$
smooth away from $0$ with Reeb vector field $\xi$ and compatible K\"ahler metric $\om$.

We can now recall the definition of test configurations for Sasakian manifolds from \cite{ColSze12}. Let $Y$ be an affine variety polarized by a Reeb vector field $\xi\in \t$, with $\t=\Lie(T)$ and $T\subset\Aut(Y)$.

\begin{defn}\label{defn:test-conf}
A T-equivariant \emph{test configuration} for $Y$ is a set of $T$-homogeneous elements $\lbrace f_1,...,f_k \rbrace $
that generate $H^0(Y,\cO_Y)$ in sufficiently high degrees together with a set of integers $\lbrace w_1,...,w_k  \rbrace$.
\end{defn}

This definition generalizes the usual test configurations for polarized manifolds or orbifolds.
Given the set of generators $\lbrace f_j \rbrace$, we can embed $Y$ into $\C^k$ and consider the $\C^*$-action on $\C^k$ with weights $(w_1,\ldots,w_k)$.  Then the flat limit $Y_0$ over $0\in\C$
of the $\C^*$-orbit of $Y$ provides a flat family of affine schemes over $\C$. Moreover, the central fiber $Y_0$ is invariant
under the $\C^*$ action defined by the weights $\lbrace w_j \rbrace$.
We will say that this test configuration is a \emph{product configuration} if $Y_0$ is isomorphic to $Y$.

Let $\upsilon$ be a generator of the $\C^*$ action on $Y_0$ defined by the weights $\lbrace w_j \rbrace$.
A weight is assigned to the test configuration, $\Fut(Y_0,\xi,\upsilon)$, the so-called Donaldson-Futaki invariant.
This weight was first defined by Futaki in the smooth case \cite{Fut83}, and then
generalized by Donaldson to the algebraic setting \cite{Don02}.  Lastly, making use of the Hilbert series, Collins and Sz\'ekelyhidi
managed to extend the definition to the Sasakian case.  See \cite[Definition 5.2.]{ColSze12}.

\begin{defn}
A polarized affine variety $(Y,\xi)$ is \emph{K-semistable} if, for every torus $T$, $\xi\in\Lie(T)$, and every $T$-equivariant test
configuration with central fiber $Y_0$, the Donaldson-Futaki invariant satisfies
\[ \Fut(Y_0,\xi,\upsilon) \leq 0 \]
with $\upsilon$ a generator of the induced $\C^*$ action on the central fiber.

It is \emph{K-polystable} if the equality holds if and only if the $T$-equivariant test configuration is a product configuration.
\end{defn}

We will only make use of the Futaki invariant for smooth test configurations, so we wont need it in its full generality.
In~\cite{BoyGalSim08} the Futaki invariant is adapted to the Sasakian case.  In this case it gives a character on the
transversely holomorphic vector fields.  A transversely holomorphic vector field is a complex vector field $X$
which projects to a holomorphic vector field on every holomorphic foliation chart (\ref{eq:fol-chart}).
We assume that $X$ is Hamiltonian, i.e. has a potential, so there is an $H\in C^\infty_b(M,\C)$ with
$\delb H =-\frac{1}{2}d\eta(X,\cdot)$.  Then the Futaki invariant is
\begin{equation}\label{eq:Fut}
\cF_{\xi,\Phi}(X)= -\int_M H(s-s_0)\, d\mu_{\eta},
\end{equation}
and only depends on the polarization and transversal complex structure.

The next lemma essentially shows that for smooth Sasakian manifolds both these
definitions are the same.
\begin{lem}\label{lem:smooth}
Let $(M_0,\xi_0,\Phi_0)$ be a polarized Sasakian manifold with corresponding affine K\"ahler cone $(Y_0=C(M_0),\xi,\om_0)$.
Suppose $\upsilon$ is the generator of a holomorphic $\C^*$-action on $Y_0$ commuting with $\xi$.
Then, there is a constant $c(n)>0$ depending only on the dimension such that
\[ Fut(Y_0,\xi,\upsilon)=c(n) \cF_{\xi,\Phi_0}(\upsilon_{M_0}). \]
\end{lem}

\begin{proof}
Since $\upsilon_{\vert M_0}$ commutes with $\xi$, it induces a transversally holomorphic vector field $\upsilon_{M_0}$ on $(M_0,\xi,\Phi_0)$.

First, suppose $\xi$ is quasi-regular.  We will show the Donaldson-Futaki invariant computed on the quotient polarized
orbifold $(X,\mathbf{L})$ is equal to (\ref{eq:Fut}).  We basically extend the computations in~\cite[Sect. 2.9]{RosTho11}.
Recall from~\cite{RosTho11} that the orbifold Riemann-Roch gives
\begin{align}
& h^0(X,\mathbf{L}^k) = a_0 k^n +a_1 k^{n-1} +\tilde{o}(k^{n-1})\label{eq:R-R1}\\
& w(H^0(X,\mathbf{L}^k)) = b_0 k^{n+1} +b_1 k^n +\tilde{o}(k^n),
\end{align}
where $w(H^0(X,\mathbf{L}^k)$ is the \emph{total weight} of $\upsilon$.
Here $\tilde{o}(k^{n-1})$ means a sum of terms in $k$ lower order than $n-1$ plus terms of the form $r(k)\delta(k)$ where
$r(k)$ is a polynomial of degree $k$ and $\delta(k)$ is periodic in $k$ of period $\Ord(X)$ and average 0.

Let $\omega\in c_1(\mathbf{L})$ be the orbifold K\"{a}hler form on $X$.  Using the induced metric on $\mathbf{K}^{orb}_X$ we have
the Ricci form $\rho\in -2\pi c_1(\mathbf{K}^{orb}_X)$.
The coefficients
\[ a_0 =\frac{1}{n!}\int_X \omega^n,\quad a_1 =\frac{1}{4\pi (n-1)!}\int_X \omega^{n-1}\wedge\rho\]
and
\[ b_0 =\frac{1}{n!}\int_X H\omega^n, \]
were computed in~\cite{RosTho11}, where $H$ is the Hamiltonian of $\upsilon$.

Let $\mathcal{O}_{\PP^1}(1)^*$ be the principal $\C^*$-bundle associated to $\mathcal{O}(1)$.  Form the associated
$(X,\mathbf{L})$-bundle
\begin{equation}\label{eq:assoc}
 (\mathcal{X},\mathbf{\cL}):= \mathcal{O}_{\PP^1}(1)^* \times_{\C^*} (X,\mathbf{L}).
\end{equation}
If $\pi:\mathcal{X}\rightarrow\PP^1$ is the projection then $\pi_* \mathbf{\cL}^k$ is the associated bundle of the
$\C^*$-representation $H^0(X,\mathbf{L}^k)$.  We have (cf.~\cite{RosTho11})
\begin{equation}
w(H^0(X,\mathbf{L}^k)) = \chi(\mathcal{X},\mathbf{\cL})-\chi(X,\mathbf{L}^k),
\end{equation}
where for $k>>1$ the right hand side is expressed using (\ref{eq:R-R1}).  This gives
\begin{equation}\label{eq:b1}
b_1 =-\frac{1}{2n!}\int_\mathcal{X} c_1(\mathbf{\cL})^n c_1(\mathbf{K}^{orb}_{\mathcal{X}}) -\frac{1}{n!}\int_X c_1(\mathbf{L})^n.
\end{equation}
It was shown by Ross and Thomas (see also~\cite{Don05}) that for associated bundle in (\ref{eq:assoc})
$c_1(\mathcal{\cL})=H\omega_{FS} +\omega$, where $\omega_{FS}$ is the
Fubini-Study metric on $\PP^1$.  Note that $\mathbf{K}^{orb}_{\mathcal{X}} =\pi^* \mathbf{K}_{\PP^1}\otimes\mathbf{\cK}$,
where $\mathbf{\cK}$ is the associated bundle on $\mathcal{X}$ to $\mathbf{K}_X$ as in (\ref{eq:assoc}).
Then the same argument gives
\[ c_1(\mathbf{K}^{orb}_{\mathcal{X}})=(f-2)\omega_{FS}-\frac{1}{2\pi}\rho, \]
where $f$ is the ``Hamiltonian'' for the action of $\upsilon$ on $\mathbf{K}_X$.
Substituting these into (\ref{eq:b1}) give
\[ \begin{split}
b_1 & =-\frac{1}{2n!}\int_{\mathcal{X}}(\omega^n +nH\omega_{FS}\wedge\omega^{n-1})\wedge((f-2)\omega_{FS} -\frac{1}{2\pi}\rho) -\frac{1}{n!}\int_X \omega^n \\
 & =-\frac{1}{2n!}\int_{\mathcal{X}}(f-2)\omega_{FS} \wedge\omega^n -\frac{n}{2\pi}H\omega_{FS} \wedge\omega^{n-1}\wedge\rho-\frac{1}{n!}\int_X \omega^n \\
 & =-\frac{1}{2n!}\int_X f\omega^n +\frac{1}{4\pi n!}\int_X sH\omega^n \\
 & =\frac{1}{4\pi n!}\int_X sH\omega^n,
\end{split} \]
where the last step follows because $f$ can be shown to be the divergence of $\upsilon$.

Recall that (cf.~\cite{RosTho11} and~\cite{ColSze12}
\[ \Fut(Y_0,\xi,\upsilon)=\frac{a_1 b_0 -a_0 b_1}{a_0}, \]
then substituting the expressions for $a_0,a_1,b_0$ and $b_1$ gives
\[ \Fut(Y_0,\xi,\upsilon) =\frac{1}{4\pi n!}\int_X H(s_0 -s)\omega^n.\]

For the general case first note that if we rescale $\tilde{\xi} =c\xi,\ c>0,$ then
\[\Fut(Y_0,\tilde{\xi},\upsilon) =c^{-(n+1)} \Fut(Y_0,\xi,\upsilon).\]
And similarly
\[\cF_{\tilde{\xi},\Phi}(\upsilon) =c^{-(n+1)}\cF_{\xi,\Phi}(\upsilon).\]
So the lemma is proved for any $\xi$ proportional to an integral element of $\t=\Lie(T)$.

In the irregular case, by \cite[Corollary 2]{ColSze12}, there is a sequence $\xi_j$ of Reeb vector fields proportional to integral vector fields on $Y_0$ such that $\xi_j \rightarrow \xi \in \t$.  Both the transversal Futaki invariant and the Donaldson-Futaki invariant depend continuously on $\xi_j$ and
the result follows at the limit.
\end{proof}

\section{Deformations and stability}\label{sec:def-stab}

We are interested in the deformation
theory of cscS Sasakian metrics.  In this section we show
how the relationship between cscS metrics and stability in the GIT sense can
be used to give an algebraic criterion for deformation of canonical Sasakian metrics.

\subsection{Deformation complexes}

We describe two complexes relevant to the deformations of a Sasakian structure $(M,\eta,\xi,\Phi_0)$ that we will consider.
The first describes deformations of the transversal complex structure of the Reeb foliation $(\mathscr{F}_\xi,\ol{J})$.

Define $\mathcal{A}^k :=\Gamma(\Lambda^{0,k}_b\otimes\nu(\mathscr{F})^{1,0})$.  We have the complex
\begin{equation}\label{eq:Dol-comp}
 0\rightarrow \mathcal{A}^{0} \overset{\ol{\partial}_b}{\longrightarrow}\mathcal{A}^{1} \overset{\ol{\partial}_b}{\longrightarrow}\mathcal{A}^{2}\rightarrow\cdots,
\end{equation}
which we denote by $\mathcal{A}^\bullet$.
This is the basic version of the complex used by Kuranishi~\cite{Kur65} whose
degree one cohomology $H^1(\mathcal{A}^\bullet)$ is the space of first order deformation of the transversal complex
structure $\ol{J}$ modulo foliate diffeomorphisms.
In~\cite{ElKacNic89} and~\cite{Gir92} it was shown that $H^1(\mathcal{A}^\bullet)$ is the tangent space to a
versal deformation space of $(\mathscr{F}_\xi,\ol{J})$ preserving the smooth foliation $\mathscr{F}_\xi$.

For the second complex, let $E^k,\ k\geq 1,$ be the kernel of the map
\[ \Lambda^{0,k}_b\otimes\nu(\mathscr{F})^{1,0}\cong\Lambda^{0,k}_b\otimes\Lambda^{0,1}\rightarrow\Lambda^{0,k+1}_b. \]
Define $\mathcal{B}^k :=\Gamma(E^k),\ k\geq 1$, and $\mathcal{B}^0 :=C^\infty_b(M,\C)$.
Note that $\mathcal{B}^1 =T_{\Phi_0}\cK$, and we define a complex $\mathcal{B}^\bullet$ by
\begin{equation}\label{eq:Sze-comp}
0\rightarrow C^\infty_b(M,\C) \overset{\cP}{\longrightarrow} T_{\Phi_0}\cK \overset{\ol{\partial}_b}{\longrightarrow}\mathcal{B}^{2}\rightarrow\cdots,
\end{equation}
where $\cP$ is (\ref{eq:Lich-com}).  The remaining maps are the same operators $\delb_b$ as above.
Then $H^1(\mathcal{B}^\bullet)$ is the space of first order deformations of $\Phi_0$ modulo the action of
$\ccG$.

There is a mapping of the complex $\cB^\bullet$ to $\cA^\bullet$.  In degree zero, this is
\begin{equation*}
\begin{array}{rcl}
C_b^\infty(M,\C) & \longrightarrow & \cA^0 \\
f & \mapsto & (\delb f)^{\sharp}
\end{array}
\end{equation*}
where $(\cdot)^{\sharp}$ denotes index raising by the transversal K\"{a}hler metric.  And for $k\geq 1$ the mapping is
the inclusion.

In many cases these two complexes give the same deformation space.
\begin{prop}\label{prop:comp-isom}
The induced map in cohomology $H^k(\cB^\bullet)\rightarrow H^k(\cA^\bullet),\ k\geq 1$, is injective if $H^{0,k}_b =0$
and is surjective if
$H^{0,k+1}_b =0$, where $H^{0,\bullet}_b$ denotes the transversal Dolbeault cohomology.
\end{prop}
\begin{proof}
Suppose $H^{0,k}_b =0$ and $\beta\in\cB^k,\ k\geq 2$.  If $[\beta]=0$ in $H^k(\cA)$, there exists a $\gamma\in\cA^{k-1}$ with
$\delb_b \gamma =\beta$.  Since $\delb_b \gamma \in\cB^k$, $\delb\skw(\gamma^\flat)=\skw(\delb_b \gamma^\flat) =0$, where
$\delb$ is the ordinary Dolbeault operator and $\gamma^\flat$ is the section of $\Lambda_b^{0,k-1}\otimes\Lambda^{0,1}_b$ obtained
from the transversal K\"{a}hler form.  By assumption there is an $\alpha\in\Gamma(\Lambda^{0,k-1}_b)$ with
$\delb\alpha=\skw(\gamma)$.  Let $\theta =\gamma-\delb_b \alpha^\sharp$.  Since
$\skw(\delb_b \alpha^\sharp)=\delb\alpha =\skw(\gamma^\flat)$, $\theta\in\cB^{k-1}$.  Since $\del_b \theta =\beta$, $[\beta]=0$ in
$H^k(\cB)$.

If $k=1$ and there exists a $\gamma\in\cA^0$ with $\delb_b \gamma =\beta$, then $\delb\gamma^\flat =0$.  There exists
an $f\in C^\infty_b(M,\C)$ with $\delb f=\gamma^\flat$.  Thus $\cP f =\delb_b \del^\sharp f =\beta$.

Suppose $\beta\in\cA^k,\ k\geq 1,\ \delb_b \beta =0$ and $H^{0,k}_b =0$.  Write $\beta=\beta_1 +\beta_2$ with respect to
\[ \cA^k =\cB^k \oplus\lambda_b^{0,k+1}. \]
Then $0=\skw(\delb_b \beta^\flat) =\skw(\delb_b \beta_2) = \delb\beta_2$, and there exists $\gamma\in\Gamma(\Lambda^{0,k}_b)$ with
$\delb\gamma =\beta_2$.  One easily sees that $\beta -\delb_b \gamma^\sharp \in\cB^k$.
\end{proof}

\subsection{Construction of the slice}\label{sec:slice}

We will construct a slice for the complex (\ref{eq:Sze-comp}) on a Sasakian manifold $(M,\eta,\xi,\Phi_0)$.
The transverse metric and the $L^2$ inner product on forms enable us to define Sobolev norms on $\cB^\bullet$ and we can define adjoint operators $\cP^*$ and $\delb_b^*$.
Then the space $H^1(\cB^\bullet)\simeq ker((\delb^* \delb) ^2 + \cP\cP^*)$ encodes infinitesimal deformations of the transverse complex structures that are compatible with $\eta$ modulo the action of $\ccG$.  This space is finite dimensional
as it is the kernel of the fourth order transversely elliptic operator
\[\Box_{\Phi_0}=(\delb_b^* \delb_b)^2 + \cP\cP^*\]
(cf.~\cite{ElKac90}).

Let $G$ be the stabilizer of $\Phi_0$ in $\cG$, then $G=\Aut(M,\eta,\xi,\Phi_0)$ and is thus compact.
This group acts linearly on $T_{\Phi_0}\cK $ and on  $H^1(\cB^\bullet)$.
The group $G$ also has a complexification $G^{\C}$ which also acts on these spaces.

\begin{prop}\label{prop:slice}
There is a holomorphic $G$-equivariant map $S$ from a neighborhood of zero $B$ in $H^1(\cB^\bullet)$ into a neighborhood of $\Phi_0$
in $\cK$ such that the $\ccG$ orbit of every integrable $\Phi$ close to $\Phi_0$ intersects the image of $S$.
Moreover, if $x$ and $x'$ lie in the same $G^{\C}$ orbit in $U$ and $S(x)\in\cK^i$, then $S(x)$ and $S(x')$ are in the same $\ccG$ orbit in $\cK$.
\end{prop}
\begin{proof}
This follows from arguments of M. Kuranishi~\cite{Kur65} adapted to the transversally complex situation and with the complex
$\cB^\bullet$ replacing $\cA^\bullet$.

Recall that $K_{\Phi_0}\subset T_{\Phi_0}\cK =\cB^1$ is an open subset.  Let $U:=K_{\Phi_0}\cap\ker\Box_{\Phi_0}$.
Then the restriction of the map (\ref{eq:chart}) to $U$ gives holomorphic map $\Psi: U\rightarrow\cK$.
We parametrize complex structures by maps $P:T^{1,0}(\Phi_0)\rightarrow T^{0,1}(\Phi_0)$, i.e. $P\in\cB^1$, as in Proposition~\ref{prop:com-chart}.

Then as in~\cite{Kur65} we construct an injective holomorphic map $\vartheta: B\rightarrow\cB^1$ on a $G$ invariant restriction $B\subset U$ so that $\varphi=\vartheta(P)$ is
in $\cK^i$ and satisfies
\[N(\varphi) =\delb_b \varphi +[\varphi,\varphi]=0 \text{  and  } \cP^* \varphi =0 \]
precisely when
\[ \mathcal{H}\bigl([\vartheta(P),\vartheta(P)]\bigr) =0,\]
where $\cH$ is the projection onto the harmonic space.  We define the slice to be $S=\Psi\circ\vartheta: B\rightarrow\cK$.
Thus the structures in $\cK^i$ are parametrized by the vanishing of a holomorphic map $\Theta: B\rightarrow H^2(\cB^\bullet)$,
\[\Theta(P) =\mathcal{H}\bigl([\vartheta(P),\vartheta(P)])\bigr.\]

The remaining properties follow from the arguments in~\cite[Lemma 6.1]{CheSun14} adapted to the transversally holomorphic situation.
\end{proof}
\begin{rmk}
The slice $S$ of the proposition is holomorphic with $B$ and $\cK$ given the $L^{2,\ell}$ topology, i.e. of Banach manifolds.
We will fix a slice for some large $\ell$.
\end{rmk}

\subsection{Reduction to finite dimensional GIT}

Assume now that $(M,\eta,\xi,\Phi_0)$ is a cscS manifold. We consider the problem of finding cscS structures on nearby
$(M,\eta,\xi,\Phi)$.  In order to get a finite dimensional moment map we will perturb the holomorphic slice in
Proposition~\ref{prop:slice}.

\begin{prop}\label{prop:map}
There exists a $G$-equivariant $C^2$ map $\hat{S}$ from a neighborhood $B$
of $0$ in $H^1(\cB^\bullet)$ to $\cK$ with $\hat{S}(0)=\Phi_0$, such that
$\mu\circ\hat{S} =(s^T -s^T_0)\circ \hat{S}$ takes value in $\Lie(G)$.

Furthermore, the $\ccG$ orbit of every integrable smooth $\Phi$ close to $\Phi_0$ intersects the image of $\hat{S}$.
If $x$ and $x'$ lie in the same $G^{\C}$ orbit in $B$ and $\hat{S}(x)\in\cK^{i}$, then $\hat{S}(x)$ and $\hat{S}(x')$ are in the
same $\ccG$ orbit in $\cK$.  Moreover, $\hat{S}$ is tangent to $S$ at 0 to first order.

Here we take $C^2$ to mean that $\hat{S}$ is $C^2$ as a map into $\cK^\ell$, the space of $L^{2,\ell}$ K-contact structures,
for some large $\ell$.  It follows that $\hat{S}$ is $C^2$ as a map from $B$ into the space of $C^m$ K-contact structures, for
$\ell >\frac{n}{2} + m$.
\end{prop}

\begin{proof}
We will perturb the map $S$ of Proposition~\ref{prop:slice} along $\ccG$-orbits to obtain our new map $\hat{S}$ with the stated
properties.  Let $B\subset H^1(\cB^\bullet)$ be a $G$-invariant neighborhood of zero.
We will identify $\Lie(\cG)$ with the corresponding subspace of basics functions.  Let $W_b^{2,k}$
be the orthogonal complement of $\Lie(G)$ in $L^{2,k}_b(M,\R)$, $\Pi_k:L^{2,k}_b(M,\R)\rightarrow W_b^{2,k}$ the orthogonal projection
on this space, and $U_k$ a small neighborhood of zero inside $W_b^{2,k}$, where we assume $k > \frac{n}{2} +2$.  Given $\phi \in U_k$ and K-contact structure
$\Phi$ with values in $L_b^{2,\ell}$, for large $\ell$, we define $F(\Phi,\phi)$ to be the structure
$f_1^*(\eta_1,\Phi_1)$, where $(\eta_1,\Phi_1)=(\eta+d^c \phi,\Phi -\xi\otimes d\phi)$ and $f_1$ is the diffeomorphism as in
Proposition~\ref{prop:GC-orbit}.  We would like to apply the implicit function theorem to
\begin{equation}
\begin{array}{cccc}
H : & B \times U_k & \rightarrow & U_{k-4} \\
 & (x,\phi)& \mapsto & \Pi_{k-4} (s^T(F(S(x),\phi))),
\end{array}
\end{equation}
which would give the desired perturbed section $\hat{S}$.  But it is not clear that $H$ is differentiable, since $f_1$ is a
$L^{2,k-2}$ diffeomorphism so $F(S(x),\phi)$ has only $L^{2,k-3}$ regularity.  So we must employ a slightly more complicated argument.

Let $\cD$ be the group of diffeomorphism of $M$.  This is an ILH Lie group~\cite{Omo70,Ebi70}, which is the inverse limit of
$\cD^k, k>\frac{n}{2}+2,$ the group of $L^{2,k}$ diffeomorphisms of $M$.
We define
\begin{equation*}\label{eq:diff-map}
\chi_i : B\times U_k \rightarrow\cD^{k-2-i},
\end{equation*}
where $\chi_i (x,\phi)$ is the diffeomorphism $f_1$ of Proposition~\ref{prop:GC-orbit} depending on the Sasakian structure
$(\eta+d^c \phi,S(x) -\xi\otimes d\phi)$.  It is known, see~\cite{EbiMar70}, that $\chi_i$ is a $C^i$ map.
In particular, $\chi:=\chi_2$ is $C^2$.  Also, we have
\begin{equation*}
\psi : B\times U_k \rightarrow\cD^{k-4},
\end{equation*}
defined by $\psi(x,\phi):=\chi(x,\phi)^{-1}$ which can be seen to be $C^2$ as follows.
Recall that $f_1$ is the $t=1$ diffeomorphism of the flow of $V_t =\Phi_t X^t_\phi$ as defined in Proposition
~\ref{prop:GC-orbit}.  Then $\psi(x,\phi)$ is the $t=1$ diffeomorphism of the flow of $-V_{1-t}$, thus $\psi$ is $C^2$.

Let $\Lie(G)$ have an orthonormal basis $\{\sigma_1,\ldots,\sigma_r \}\subset C_b^\infty$.  Then define
\[ \Pi_k =Id -\Pi_k^G ,\]
where $\Pi_k^G : L_b^{2,k} \rightarrow\Lie(G)$ is the orthogonal projection,
\[ \Pi_k^G(f) = \sum_{j=1}^r \sigma_j \langle f,\sigma_j \rangle_{L^2}. \]

Instead of $H$ we consider
\begin{equation}\label{eq:slice-map}
\begin{array}{cccc}
F : & B \times U_k & \rightarrow & U_{k-4} \\
 & (x,\phi)& \mapsto & \Pi_{k-4} \bigl(\psi(x,\phi)^*\Pi_{k-4} \bigr) s^T(S(x),\phi),
\end{array}
\end{equation}
where $s^T(S(x),\phi)$ is the transverse scalar curvature of $(\eta+d^c \phi,S(x) -\xi\otimes d\phi)$.
We have $\psi(x,\phi)^*\Pi_{k-4} = Id - \psi(x,\phi)^* \Pi_{k-4}^G$ and
\begin{equation*}
\begin{split}
\psi(x,\phi)^* \Pi_{k-4}^G (f) & = \psi(x,\phi)^* \circ \Pi_{k-4}^G \circ \chi(x,\phi)^* f \\
                               & = \psi(x,\phi)^* \bigl(\sum_{j=1}^r \sigma_j \langle\chi(x,\phi)^* f,\sigma_j \rangle\bigr) \\
                               & = \sum_{j=1}^r \psi(x,\phi)^* \sigma_j \langle f,\psi(x,\phi)^* \sigma_j \rangle_{\eta +d^c \phi}, \\
\end{split}
\end{equation*}
is easily seen to be $C^2$ on $B\times U_k$.  Because $\sigma_j \in C^\infty_b$ we have that $\psi(x,\phi)^* \sigma_j$ is
a $C^2$ mapping from $U\times U_k$ to $L^{2,k-4}_b$ (cf.~\cite{EbiMar70}).  Therefore $F$ is $C^2$.

The derivative of $F$ with respect to $U_k$ at the origin is
\begin{equation*}
DF_{(0,0)} = \cP^* \cP : W^{2,k} \rightarrow W^{2,k-4},
\end{equation*}
which is an isomorphism for all $k>\frac{n}{2} +2$.  Applying the implicit function theorem for $k\gg 1$,
we get a $C^2$ mapping
\[ U\ni x \mapsto \phi(x)\in L^{2,k}_b(M), \]
so that $F(x,\phi(x))=0$.

After possibly shrinking $B$ this implies that on $B$
\begin{equation}\label{eq:scal-van}
\bigl(\psi(x,\phi(x))^*\Pi_{k-4} \bigr) s^T(S(x),\phi(x))=0.
\end{equation}
Since $\phi(x)\in L^{2,k}_b(M)$ for each $x\in B$, we have $\psi(x,\phi(x))\in\cD^{k-2}$ for all $x\in B$.
From (\ref{eq:scal-van}) we have
\[ s^T(S(x),\phi(x)) \in\im\Bigl[\psi(x,\phi(x))^*\Pi^G \Bigr], \]
where the right hand side is a finite dimensional spaces of $L^{2,k-2}$ function.
Note that on any local transverse chart of the Reeb foliation $\mathscr{F}_\xi$ as in (\ref{eq:fol-chart}) the
map $\phi\mapsto s^T(S(x),\phi)$ is elliptic.  By well known elliptic regularity $\phi(x)\in L^{2,k+2}_b(M)$.  But
then $s^T(S(x),\phi(x))\in L^{2,k}_b(M)$, so elliptic regularity implies $\phi(x)\in L^{2,k+4}_b(M)$.
Continuing, we see that $\phi(x)\in C^\infty_b(M)$ for each $x\in B$.

It follows that
\begin{equation}\label{eq:slice-def}
\hat{S}(x) =(\eta,\Phi(x)) :=\chi(x,\phi(x))^* \bigl(\eta +d^c \phi(x), S(x)-\xi\otimes d\phi(x) \bigr)
\end{equation}
is section of $\cK$, which is $C^2$ as a map from $B$ to $\cK^{\ell}$ for $\ell=k-5$.
Note that at the cost of shrinking $U$, one can choose $\ell$ arbitrarily large.
This is also a $C^2$ map from $U$ to $C_b^m(M)$ for $\ell>\frac{n}{2}+m$ by the Sobolev
embedding theorem.

The statements about the orbits of $\ccG$ follow from the same properties of the slice $S$ of Proposition
~\ref{prop:slice} since $\hat{S}$ is constructed by deforming $S$ along $\ccG$ orbits.

The $G$-equivariance of $\hat{S}$ can be easily checked from that of $S$ and the other maps involved.
Suppose $A\in T_{\Phi_0}\cK$ is tangent to $\hat{S}$.  We have
\[ \langle ds^T (A), H\rangle = g_{\cK}(\Phi_0 \cP_{\Phi_0}(H), A)=0,\quad\text{for all } H\in C_b^\ell ,\]
since $ds^T (A)\in\Lie(G)$ and $\cP_{\Phi_0}(H) =0$ for $H\in\Lie(G)$.
Thus $A$ is orthogonal to $\im(\Phi_0 \cP_{\Phi_0})$ and one can check by differentiating (\ref{eq:slice-def})
that the component of $\im d\hat{S}(0)$ not tangent to $S$ is in $\im(\Phi_0 \cP_{\Phi_0})$.
\end{proof}

Let $U\subset H^1(\cB^\bullet)$ be a G-invariant neighborhood as above, then
the pullback of the moment map $\mu= s^T -s^T_0$ to $U$ by $\hat{S}$ is a moment map for the $G$-action on $U$ with
respect to the pullback of the symplectic form $\Omega_{\cK}$ on $\cK$ by $\hat{S}$, denoted $\Omega$.
Note that the pulled back moment map $\mu$ on $U$ is $C^2$ and $\Omega$ is $C^1$.
We have reduced the problem of finding zeros of the moment map $\mu$ to
a finite dimensional Hamiltonian system $(U,G,\Omega,\mu)$.  But note that this problem is slightly complicated by the
fact that $\Omega$ is not K\"{a}hler.

\begin{prop}\label{prop:finiteGIT}
Suppose that $x\in U$, after possibly shrinking $U$, is polystable for the $G^{\C}$-action on $H^1(\cB^\bullet)$.
Then there is $y$ in the $G^{\C}$-orbit of $x$ such that
$s^T(\hat{S}(y))-s^T_0=0$. If in addition $\hat{S}(x)$ is integrable, then the corresponding cscS manifold $(M,\eta,\hat{S}(y))$ is
a deformation of $(M,\eta,\hat{S}(x))$ as in (\ref{eq:trans-def2}).
\end{prop}

\begin{proof}
Let $\Omega$ be the restriction of the symplectic form $\Omega_{\cK}$ of $\cK$ to $U\subset H^1(\cB^\bullet)$ via
$\hat{S}$.  Thus the restriction of the moment map $\mu: U \rightarrow\Lie(G)$ is a moment map for $G$.
Also, let $\Omega_0$ be the linear symplectic form on $H^1(\cB^\bullet)$ induced by $\Omega$ at the origin.
And write
\begin{equation*}
\nu : H^1(\cB^\bullet)\rightarrow\Lie(G)
\end{equation*}
for the corresponding moment map for the flat K\"{a}hler structure $(\Omega_0,J)$, where $J$ is the vector space
complex structure on $ H^1(\cB^\bullet)$.  Let $\sigma_x :\Lie(G)\rightarrow T_x U$ be the infinitesimal action,
then
\[ \langle\nu(x),\xi\rangle =\frac{1}{2}\Omega_0(\sigma_x \xi,x),\quad\text{for }\xi\in\Lie(G).\]

We have $\mu(0)=0$ by assumption, and $d\mu_0 =0$ because $0\in U$ is a fixed point of the $G$-action.
It is also routine to check that
\[ \frac{d^2}{dt^2} \mu(tx)|_{t=0} =2\nu(x).\]
Therefore by Taylor's theorem we have
\begin{equation}\label{eq:moment-asym}
\mu(tx)-\nu(tx) = o(t^2),
\end{equation}
with $\underset{t\rightarrow 0}{\lim}\frac{|\mu(tx)-\nu(tx)|}{t^2} =0$ converging to $0$ uniformly in $x\in U$.
Given $\delta >0$ and $C>0$ we may shrink $U$ so that
$Ct^{-2}|\mu(tx)-\nu(tx)| <\delta$.

If $x\in U$ is polystable for the $G^{\C}$-action, then by the Kempf-Ness theorem there is a zero $x_0$ of $\nu$
in the $G^{\C}$ orbit of $x$.  Since this is given by minimizing the norm over $G^{\C}\cdot x$, the zero
will be in $U$.

For $x\in U\subset H^1(\mathcal{B}^\bullet)$ let $K_x \subset G$ the stabilizer of $x$ and $\mathfrak{k}_x$ its Lie algebra.  We
have $K_{tx}=K_x$.  Then for $\xi\in\mathfrak{k}_x$
\[ \frac{d}{dt}\langle\mu(tx),\xi\rangle =\Omega_{tx}(\sigma_{tx}(\xi),x)=0. \]
Thus for all $x\in U$ we have $\mu(x)\in\mathfrak{k}_x^{\perp}$.

We also shrink $U$ so that $\Omega(X, JX)\geq\frac{1}{2}\Omega_0(X,JX),\ \forall X\in TU$.

Define functions $h=\|\nu\|^2$ and $f=\|\mu\|^2$ on $U$, where the norms are given by invariant metrics on $\Lie(G)$.
Note that $f$ is the Calabi functional.
We have $dh_x(w) =2\Omega_0(\sigma_x \nu(x),w),\ w\in T_x U$.

Let $x_0$ be a point with $\nu(x_0)=0$.  We consider the restriction of $f$ and $h$ to the orbit $G^{\C}\cdot x_0$.
Note that both $f$ and $h$ are invariant under $G$.
Let $R\subset G^{\C}\cdot x_0$ be a local slice through $x_0$ for the $G$ action on $G^{\C}\cdot x_0$ invariant under
$K_{x_0}$.  Clearly $dh_{x_0} =0$.  The Hessian at of $h$ at $x_0$ is
\begin{equation}
d^2_{x_0}h(v,w) = 2\sum_{j=1}^m \Omega_0(\sigma_{x_0} e_j,v)\Omega_0(\sigma_{x_0} e_j,w),
\end{equation}
where $\{e_1,\ldots, e_m \}$ is an orthonormal basis of $\mathfrak{k}^{\perp}_{x_0}$.  If $v\in T_{x_0}R$ then
$v=J\sigma_{x_0} \eta$ for $\eta\in\mathfrak{k}^{\perp}_{x_0}$.  And
\[ d^2_{x_0} h (v,v) =\sum_{j=1}^m g_0(\sigma_{x_0} e_j ,\sigma_{x_0}\eta)^2 >0, \]
for $v\neq 0$.

There exists constants $\epsilon>0$ and $\delta_2 >\delta_1>0$ so that $h<\epsilon^2$ on $B_{\delta_1}(x_0)$,
ball of radius $\delta_1$ in $R$,
while $h>9\epsilon^2$ on $A=B_{2\delta_2}(x_0) \setminus B_{\delta_2}(x_0)$.
We will identify the slice $R$ and the above neighborhood in the obvious way when $x_0$ is replaced with $tx_0,\ 0<t<1$.
By shrinking $U$ we may assume $\|\mu(tx) -\nu(tx)\|\leq\frac{1}{2}\epsilon t^2$, which implies that
\[ \|\mu(tx)\| \geq \frac{5}{2}\epsilon t^2\quad\text{on }A \]
and
\[ \|\mu(tx)\| \leq\frac{3}{2}\epsilon t^2\quad\text{on }B_{\delta_1}.\]
So on the slice $R$ through $tx_0$ there exists a minimum of $f$ at $x\in B_{\delta_2}$.
Since $0=df_x (v)=2\Omega(\sigma_x \mu(x),v)$ on $G^{\C}\cdot tx_0$ we have
\[\begin{split}
0 = 2\Omega(\sigma_x \mu(x),J\sigma_x \mu(x)) & \geq \Omega_0(\sigma_x \mu(x),J\sigma_x \mu(x)) \\
                                              & =g_0(\sigma_x \mu(x),\sigma_x \mu(x)).\\
\end{split}\]
Thus $\mu(x)\in\mathfrak{k}_x$ and $\mu(x)=0$ since we have $\mu(x)\in\mathfrak{k}_x^\perp$.

\end{proof}

This proposition gives an algebraic, finite dimensional way to obtain deformations of cscS metrics.

\begin{rmk}
In the K\"{a}hler case, i.e. $(M,\eta,\xi,\Phi_0)$ is a regular Sasakian manifold, gives the local moduli of
csc metrics.  The proposition proves that locally every polystable orbit contains a csc metric.  The results of X. Chen and
S. Sun~\cite{CheSun14} show the necessity and uniqueness.  More precisely, they show that only polystable orbits contain csc metrics
and they are unique up to the action of $\cG$.  Thus the local moduli is given by a neighborhood of zero in the GIT quotient
$H^1(\cB^\bullet){/\!/}\ccG$.  The necessity of polystability and uniqueness in the csc Sasakian case is still open.
\end{rmk}

\begin{rmk}
 All these results extends to the extremal setting, using relative versions of
 group actions containing the extremal vector field. See \cite{RolTip12}.
\end{rmk}

\subsection{Deformations and K-stability}

In this section, we show that if a Sasakian manifold is obtained by a small deformation of the transverse complex structure of a cscS
manifold, then $K$-polystability is a sufficient condition to admit a cscS metric.  We also show that a small deformation of
a cscS metric is K-semistable.

\begin{thm}\label{thm:csc-def}
Let $(M,\eta,\xi,\Phi_0)$ be a cscS manifold and $(M,\eta,\xi,\Phi)$ a nearby Sasakian manifold with
transverse complex structure $\bar J$. Then if $(M,\eta,\xi,\Phi)$ is K-polystable, there is a constant scalar curvature Sasakian structure in the space $\cS(\xi, \bar J)$.
\end{thm}

\begin{proof}
Consider the map $\hat{S}$ from Proposition~\ref{prop:map}.
If $\Phi$ is close enough to $\Phi_0$, then up to the action of $\ccG$ there is $x\in U\subset H^1(\cB^\bullet)$
such that $\hat{S}(x)=\Phi$.
Assume for the moment that the K-polystability of $(M,\eta,\Phi)$ implies the polystability of $x$ under the action of $G^{\C}$ in $U$,
where $G$ is the stabilizer of $\Phi_0$ in $\cK$ under the $\cG$ action.  Then by proposition~\ref{prop:finiteGIT} the result follows.
What remain to be shown is that if $x$ is not polystable in $U$, then $(M,\eta,\Phi)$ is not K-polystable.

Consider a one parameter subgroup
\[ \rho: \C^* \rightarrow G^{\C} \]
such that
\[ \exists \lim_{\lambda\rightarrow 0}\rho(\lambda)\cdot x=x'\in U \]
with $x'$ polystable.  Note that we can assume $\rho(S^1)\subset G$, the stabilizer of $\Phi_0$.
We will build a destabilizing test configuration for $(M,\eta,\Phi)$ from this one parameter subgroup.

Note that the polarized affine cones defined by $(M,\xi,\hat{S}(x))$ and $(M,\xi,S(x))$ are biholomorphic.
In fact by the construction of $\hat{S}$ the Sasakian structures differ as in Proposition~\ref{prop:GC-orbit}.
So it is enough to construct a destabilizing test configuration for $(M,\xi,S(x))$.

We have a holomorphic map
\[ F:\Delta \rightarrow \cK^i,\]
where $\Delta\subset\C$ is the unit disk, $F(\lambda)=S(\rho(\lambda)\cdot x)$ for $\lambda\in\Delta\setminus\{0\}$, and $F(0)=S(x')$.  Our test configuration will be $\cY=Y\times\Delta$ as a smooth
manifold with $Y$, the cone for $S(x)$.  The complex structure on $Y\times\{t\}$ is given by $F(t)$, and the rest of
the complex structure is given by the holomorphic map $F$.  The $S^1$-action on $\cY$ is given by
\[ \rho(\tau)(y,t)=(\rho(\tau)y,\tau t),\]\
and extends to a holomorphic $\C^*$-action.
Furthermore, the real torus $T$ generated by $\xi$ acts fiber-wise on
$\cY$.

To see this gives a test configuration as in Definition~\ref{defn:test-conf} choose a subgroup $S^1 \subset T$ generated by
$\zeta$ in the Reeb cone.  Then quotienting by the $\C^*$ generated by this action gives an orbifold test configuration
$(\XX,\cL)$ where $\cL$ is a positive orbifold bundle, with the $\C^*$-action generated by $\rho$.
As in~\cite{RosTho11} we can embed $\XX$ in a weighted projective bundle $\PP(\oplus_{i} (\pi_* \cL^{w_i})^*)$ over $\Delta$,
for large enough $w_i \in\N$, where the summands are preserved by $T$.
Then the $\rho$ action on $(\XX_0,\cL_0)$ induces a $\C^*$-action on
$V=\oplus_i H^0(\XX_0,\cL^{w_i})$.  This induces a diagonal action on
$V\times\C$ inducing $\rho$ on $\cY\subset V\times\C$.
Then the $\{f_i\}$ come from picking $T$-homogeneous elements of the summands of $V$, and the $\{w_i\}$ are the weights of
$\rho$ on these elements.  Thus we have $\TT$-equivariant test configuration $\cY$ for $(Y_1,\xi)$ with central fiber $Y_0$.

Since $x'$ is polystable, $Y_0=\hat{S}(x')$ has a cscS structure in its $G^{G}$ orbit from Proposition~\ref{prop:finiteGIT}.
From Lemma~\ref{lem:smooth}, the Donaldson-Futaki invariant of this test configuration is the usual transversal Futaki invariant which vanishes because $Y_0$ admits a cscS structure.  Moreover, the stabilizer of $\Phi'=S(x')$ in $\cG$ is strictly greater
than that of $\Phi$ because $x'$ is in the closure of the orbit of $x$. Then the automorphism group of $Y_1$ is smaller than the automorphism group of $Y_0$ and this test configuration is not a product. Thus $(M,\eta,\Phi)$ is not
$K$-polystable, which ends the proof.
\end{proof}

\begin{thm}\label{thm:K-semi-def}
Let $(M,\eta,\xi,\Phi_0)$ be a cscS manifold.  Then any small deformation $(M,\eta,\xi,\Phi)$ which is Sasakian is K-semistable.
\end{thm}
\begin{proof}
In~\cite{ColSze12} the following inequality of Donaldson~\cite{Don05} is proved for a Sasakian manifold $(M,\eta,\xi,\Phi)$ with
polarized cone $(Y,\xi)$.  For any test configuration of $(Y,\xi)$ we have
\begin{equation}\label{eq:Cal-bound}
\inf_{g\in\cS(\xi, \bar J)}(\Cal(g))^{\frac{1}{2}}\|\upsilon\|_{\xi} \geq c(n)\Fut(Y_0,\xi,\upsilon)
\end{equation}
where $c(n)>0$ is a constant that only depends on the dimension n.  See~\cite{ColSze12} for the definition of the norm
$\|\upsilon\|_{\xi}$.

We may assume that $\Phi =\hat{S}(x)$ for some $x\in U$.  If $x$ is in a polystable orbit then it admits a csc representative
in $\cS(\xi, \bar J)$ and is thus K-semistable.  Thus we may assume that $x$ is in a non-polystable orbit under the action
of $G^{\C}$.  If $(Y,\xi)$ is the polarized cone  of $\hat{S}(x)$ then there is a test configuration $\cY$ with
$\pi^{-1}(1)=Y_1 =(Y,\xi)$ and central fiber $Y_0$ corresponding to a polystable $x' \in U$.  Since $\cY$ is smooth $\pi:\cY\rightarrow\Delta\subset\C$ is a submersion, and by Ehresmann's theorem there is a diffeomorphism
\[ F:C(M)\times\Delta \rightarrow\cY , \]
where $C(M)=(Y,\xi)$.  Furthermore we may take $F$ to be $T$-equivariant, where $T$ is the torus generated by $\xi$.
Then $F$ defines a smooth family of Sasakian structures on $M$ denoted by $M_z,\ z\in\Delta$.
By Proposition~\ref{prop:finiteGIT} the central fiber $Y_0 =C(M_0)$ where $M_0$ is a Sasakian manifold
$(M,\eta_0,\xi,\Phi_0)$ with a cscS deformation in $\cS(\xi, \bar J_0)$.  That is, there is a $T$-invariant
$\phi\in C^\infty_b(M_0)$ so that $\tilde{\eta}=\eta_0 +d^c \phi$ and $\tilde{\omega}^T =\omega^T +\frac{1}{2}dd^c \phi$ defines
a cscS structure.  Using $F$ we see that the deformed structure $(\tilde{\eta}_z,\xi,\tilde{\Phi}_z,\tilde{g}_z)$ of $M_z$
with $\tilde{\eta}_z=\eta_z +d^c \phi$ is a Sasakian structure for $|z|<\epsilon$ for some small $\epsilon >0$.
Then we see that $(\tilde{\eta}_z,\xi,\tilde{\Phi}_z,\tilde{g}_z),\ |z|<\epsilon,\ z\neq 0,$ is a family of Sasakian structures in
$\cS(\xi, \bar J)$ with $\lim_{z\rightarrow 0} \Cal(\tilde{g}_z) =0$, since the metrics $\tilde{g}_z$ converge uniformly
to the csc structure $(\tilde{\eta}_0,\xi,\tilde{\Phi}_0,\tilde{g}_o)$.  The Theorem now follows from (\ref{eq:Cal-bound}).
\end{proof}

\section{Examples}\label{subsec:examples}

We give some examples of cscS manifolds for which the previous results give nontrivial cscS deformations giving new cscS metrics.
We also get some examples which are K-semistable but not K-polystable.

\subsection{Toric Sasakian manifolds}

We give some toric 5-dimensional cscS manifolds with nontrivial deformations.  The examples are quasi-regular and
are given by explicit cones over two dimensional fans describing orbifold toric surfaces.  But modifications of
the following arguments using nonrational polytopes as in~\cite{Abr10} and~\cite{Leg11} should give irregular examples also.

\begin{prop}\label{prop:toric}
Let $(M^{2m+1},g,\eta,\xi,\Phi)$ be a compact toric Sasakian manifold.  Then $H^1(M,\R)=0$ and
the basic Hodge numbers satisfy $h^{0,k}_b =0$ for $k\geq 1$.
\end{prop}
\begin{proof}
Let $\gamma\in H_b^{0,k}$ be harmonic.  Then $\ol{\gamma}\in\Gamma(\Lambda_b^{k,0})$ is a basic, transversely holomorphic form.
The torus $T^{m+1}$ acts on $M$ preserving the foliation an transversely holomorphic structure.
Let $\{e_i\}_{i=1,\ldots,m+1}$ be a basis of $\t =\Lie(T)$ with $e_{m+1} =\xi$.
At any point of $M$, in an open dense set, linear combinations of $\{e_i\}_{i=1,\ldots,m}$ span the transversal holomorphic
tangent space.  Let $f_1 \wedge\cdots\wedge f_k \in\Gamma(\Lambda^k TM)$ where the $f_i$ are linear combinations of
$\{e_i\}_{i=1,\ldots,m}$.  Then $\ol{\gamma}(f_1 \wedge\cdots\wedge f_k)$ is constant because it basic and transversely holomorphic.
But there exists strata where $f_1 \wedge\cdots\wedge f_k =0$ if $k\geq 1$, so this constant must be zero.

Suppose $\beta\in\cH^1_g$ is harmonic.  Since $\beta$ is
invariant under $T$, $\cL_\xi \beta =d \xi\contr\beta=0$.  Thus
$\xi\contr\beta=c$ and an easy argument shows $c=0$.  Thus $\cH^1_g \cong H^1_b =0$, by the above.
\end{proof}

If follows from Proposition~\ref{prop:comp-isom} that if $M$ is a toric Sasakian manifold
\[ H^k(\cA^\bullet) =H^k(\cB^\bullet),\quad\text{for }k\geq 1.\]

A toric orbifold surface is described by a fan $\Sigma$ with 1-dimensional cones
\[ \Sigma^{(1)}=\lbrace u_1,\ldots,u_d \rbrace, \]
where $u_i \in\Z^2, i=1,\ldots, d$ are not necessarily primitive.  Then an orbifold polarization is given by a polytope
$\Delta$ defined by
\[ \langle u_i, x\rangle \leq \lambda_i,\quad \lambda_i \in\Z,\ i=1,\ldots,d.\]
Denote by $X_\Delta$ the polarized surface.
Then the vectors $w_i =(u_i, \lambda_i)\in\Z^3$ span a cone in $\Z^3$ defining the polarized affine toric variety $(Y,\xi)$,
with polarization $\xi =(0,0,1)$ giving the cone over a toric Sasakian manifold.

As in~\cite{Don02} we define a measure $d\sigma$ on $\del\Delta$ which on the edge defined by
$\langle u_k, x\rangle \leq \lambda_k$ is
\[ d\sigma:= \frac{1}{|u_k|} d\sigma_0,\]
where $d\sigma_0$ is the Lebesgue measure.  The average of the scalar curvature $S_0$ of a K\"{a}hler metric is an invariant
of the polarization, and is given by
\begin{equation}\label{eq:av-toric}
S_0 =\frac{\int_{\del\Delta} d\sigma}{\int_{\Delta} d\mu}.
\end{equation}

S. Donaldson~\cite{Don09} proved that K-polystability with respect to toric degenerations implies the existence of a constant
scalar curvature K\"{a}hler metric.  In general K-polystability is difficult to check, but B. Zhou and X. Zhu~\cite{ZhouZhu08a}
gave a simple condition implying K-polystability relative to toric degenerations.

Suppose the Futaki invariant of $X_\Delta$ vanishes.  This is equivalent to the vanishing of
\[ L(\theta):= \int_{\del\Delta} \theta\, d\sigma -S_0 \int_{\Delta} \theta\, d\mu \]
for all affine linear functions $\theta$.  If
\begin{equation}\label{eq:ZhouZhu}
 S_0 <\frac{n+1}{\lambda_i},\ i=1,\ldots d,
\end{equation}
then $X_\Delta$ is K-polystable for toric degenerations.  Donaldson's result then implies that $X_{\Delta}$ admits a constant
scalar curvature K\"{a}hler metric.

Consider $\C\PP^1\times\C\PP^1$ with the $\Z_q$-action
\[ \alpha\cdot ([x_1,y_1],[x_2,y_2])=([\alpha x_1,y_1],[\alpha x_2,y_2]),\ \text{for}\ \alpha\in\mu_q, \]
where $\mu_q \subset\C^*$ is the group of q-th roots of unity.
Then $\C\PP^1\times\C\PP^1 /\Z_q$ is a toric surface with fan $\Sigma_q$ given by
\[ \Sigma^{(1)} =\lbrace e_1, e_1 +qe_2, -e_1, -e_1 -qe_2 \rbrace. \]

For simplicity we restrict to $q=3$, shown in Figure~\ref{fig: sur-quo}.

\begin{figure}[htbp]\label{fig: sur-quo}
\centering
\psset{unit=0.85cm}
\begin{pspicture}(-4.5,-7)(4.5,0.5)
$$
\xymatrix @M=0mm{
\bullet & \bullet & \bullet & \bullet & \bullet \\
\bullet & \bullet & \bullet & \bullet & \bullet \\
\bullet & \bullet & \bullet & \bullet & \bullet \\
 \bullet & \bullet & \bullet\ar[r]\ar[ruuu]\ar[l]\ar[lddd]& \bullet & \bullet \\
\bullet & \bullet & \bullet & \bullet & \bullet \\
\bullet & \bullet & \bullet & \bullet & \bullet \\
\bullet & \bullet & \bullet & \bullet & \bullet
}
$$
\end{pspicture}
\caption{$\C\PP^1\times \C\PP^1/\Z_3$}
\end{figure}

Then the toric minimal resolution of $\C\PP^1\times\C\PP^1 /\Z_3$ has $\Sigma^{(1)}$ given by
$u_1 =(1,0), u_2 =(1,1), u_3 =(1,2), u_4 =(1,3), u_5 =(0,1), u_6 =-u_1, \ldots, u_{10} =-u_5$.

If one chooses $(\lambda_1, \ldots, \lambda_{10}) =(9,8,8,10,6,9,8,8,10,6)$, then one gets a polytope with this fan and
one computes using (\ref{eq:av-toric}) that

\[ S_0 =\frac{12}{15}+\frac{8\sqrt{10}}{345}+\frac{2\sqrt{5}}{115} +\frac{2\sqrt{2}}{115} =0.24115\ldots. \]
The inequalities (\ref{eq:ZhouZhu}) are clearly satisfied.  Then the toric affine variety $Y$ defined by the $w_i$ as above,
$w_1 =(1,0,9),w_2 =(1,1, 8),\ldots, w_{10} =(0,-1, 6)$, is the polarized cone over a cscS manifold $(M,g,\xi)$.

Let $\hat{\Sigma}$ be the cone spanned by the $w_i, i=1,\ldots,d$.  We are interested in the deformations of $Y_{\hat{\Sigma}}$
preserving $\xi=(0,0,1)$.  Consider the weight space decomposition of $H^1(\mathcal{B}) =H^1(\mathcal{A})$ under
$\TT^3$
\[ H^1(\mathcal{A}) =\bigoplus_{R\in\mathcal{W}} H^1(\mathcal{A})(R),\]
where clearly each non-trivial term has $R$ vanishing on $\xi$.  The work of
N. Ilten and R. Vollmert~\cite{IltVol12,Vol11} constructs deformations corresponding to a homogeneous
component $H^1(\mathcal{A})(R)$ from \emph{admissible} Minkowski decompositions of
$\hat{\Sigma}_R = \hat{\Sigma}\cap\lbrace R=1\rbrace$.  With $R=e_1^*$ for this example,
$\hat{\Sigma}_{e_1^*}$ has a three term admissible Minkowski decomposition
\[ \hat{\Sigma}_{e_1^*} =\Delta_0 +\Delta_1 +\Delta_2 \]
giving a 2-parameter deformation spanning $H^1(\mathcal{A})(e_1^*)$.  Similarly, we have a 2-parameter deformation
spanning $H^1(\mathcal{A})(-e_1^*)$.

\begin{prop}\label{prop:obst-toric}
Let $M$ be a toric Sasakian 5-manifold, then $H^2(\mathcal{B})=H^2(\mathcal{A})=0$.
\end{prop}
\begin{proof}
Note that $H^2(\mathcal{A})= H^0_{\del_b}(\Lambda^{2,0}\otimes\Lambda^{1,0})$,
while the latter consists of transversely holomorphic sections and is easily seen to be zero by evaluating a section
on holomorphic vector fields generated by $\TT^3$.
\end{proof}

Thus we have a 4-parameter family of deformations
\[ H^1(\mathcal{A})(e_1^*)\oplus H^1(\mathcal{A})(-e_1^*)\]
which are integrable by Proposition~\ref{prop:obst-toric}.  The polystable elements with respect to $\TT^3$, the complexification
of the identity component of the isometry group of $(M,g,\xi)$, are $(0,0)$ and $(x_1,x_2)$ with $x_1 \neq 0$ and $x_2 \neq 0$.
The latter give cscS metrics with a $T^2$ group of isometries.  The remaining orbits $(x_1, 0)$ and $(0,x_2)$
give examples that are K-semistable but not K-polystable by Theorems~\ref{thm:csc-def} and~\ref{thm:K-semi-def}.

In the second example we consider a cone over a partial resolution of $\C\PP^1\times\C\PP^1 /\Z_3$.
Let $\hat{\Sigma}$ be the cone spanned by
$w_1 =(1,0,9),w_2 =(1,1,7), w_3 =(1,3,10),w_4 =(0,1,6),w_5=(-1,0,9),w_6 =(-1,-1,7),w_7 =(-1,-1,10),w_8 =(0,-1,6)$.
It is easily seen to be strongly convex with each pair of successive rays extending to a basis of $\Z^3$.  Thus
$Y_{\hat{\Sigma}}$ is a toric affine cone smooth away from the vertex and the cone over a toric Sasakian 5-manifold $M$.

Using (\ref{eq:av-toric}) we compute
\[ S_0 =\frac{20}{223} +\frac{6\sqrt{10}}{223} +\frac{14\sqrt{2}}{223} =0.26355\ldots \]
and the inequalities (\ref{eq:ZhouZhu}) are satisfied.  And $Y_{\hat{\Sigma}}$ is the cone over a toric cscS manifold
$(M,g,\xi)$.

Using the same arguments as above and the fact that $\hat{\Sigma}_{e_1^*}$ has a two term admissible Minkowski decomposition,
we have a 2-parameter family of deformations
\[ H^1(\mathcal{A})(e_1^*)\oplus H^1(\mathcal{A})(-e_1^*).\]
The polystable elements with respect to $\TT^3$ are $(0,0)$ and $(x_1,x_2)\in H^1(\mathcal{A})(e_1^*)\oplus H^1(\mathcal{A})(-e_1^*)$ with $x_1 x_2 \neq 0$.  The remaining orbits are not K-polystable but are K-semistable.

For the third example we consider a non-regular modification of the first example.  Define $\hat{\Sigma}$ to be the cone
spanned by $w_1 =(1,0,9),w_2 =(1,1,8),w_3 =(1,2,8),w_4 =(1,3,10),w_5 =(0,3,10),w_6 =(-1,0,9),w_7 =(-1,-1,8),w_8 =(-1,-2,8),
w_9 =(-1,-3,10),w_{10} =(0,-3,10)$.  We compute
\[ S_0 =\frac{32}{265} +\frac{8\sqrt{10}}{795} +\frac{6\sqrt{5}}{265} +\frac{6\sqrt{2}}{265} =0.23522\ldots,\]
and inequalities (\ref{eq:ZhouZhu}) are easily seen to be satisfied, so we have a cscS metric.
In this example $\hat{\Sigma}_{e_1^*}$ has a three term admissible Minkowski decomposition, and we get a 4-parameter
family of deformations
\[ H^1(\mathcal{A})(e_1^*)\oplus H^1(\mathcal{A})(-e_1^*).\]
Again, the polystable orbits are $(0,0)$ and $(x_1,x_2)$ with $x_1 \neq 0$ and $x_2 \neq 0$.

One can construct an unlimited number of examples by taking cones over partial resolutions of $\C\PP^1\times\C\PP^1 /\Z_q$
as in these examples.

\subsection{3-Sasakian manifolds}

A 3-Sasakian manifold is a Riemannian manifold $(M,g)$ admitting three Sasakian structures $(\eta_i,\xi_i,\Phi_i), i=1,2,3$
with $[\xi_i,\xi_j] =-\varepsilon^{ijk}\xi_k$, where $\varepsilon^{ijk}$ is antisymmetric and $\varepsilon^{123}=1$.
Thus $\{\xi_1,\xi_2,\xi_3\}$ generate the Lie algebra $\LSp(1)$ of $\Sp(1)$.  Thus $\Sp(1)$ acts by isometries on
$(M,g)$ rotating the Sasakian structures.  See~\cite{BoyGal99} and~\cite{BoyGal08} for details.

We remark that a 3-Sasakian structure on $(M,g)$ is equivalent to a hyperk\"{a}hler structure on
$(C(M),\ol{g}=dr^2 +r^2 g)$.  Thus $M$ must be of dimension $n=4m-1$ and $(M,g)$ is Sasaki-Einstein with Einstein constant
$4m-2$.

Fix a Sasakian structure, say $(\eta,\xi,\Phi):= (\eta_1,\xi_1,\Phi_1)$.  Since $(M,g)$ has positive Ricci curvature, standard
vanishing theorems and Proposition~\ref{prop:comp-isom} imply the following.
\begin{prop}
For $(M,\eta,\xi,\Phi)$ we have $H^1(\cB)=H^1(\cA)$ and $H^2(\cA)=0$.
\end{prop}

The element $\tau =e^{\frac{\pi}{2}\mathbf{j}}\in\Sp(1)$ acts on $(\eta,\xi,\Phi)$ by $\tau\cdot(\eta,\xi,\Phi)=(-\eta,-\xi,-\Phi)$.
We consider deformations equivariant with respect to $\tau$.  We have a conjugate linear isomorphism
\[ \tau_* : H^1(\cA)\rightarrow H^1(\cA).\]
Since $\tau^2 =\Id$, this is a real structure and we define $\re H^1(\cA)$ to be the subspace fixed by $\tau_*$.

Let $G=\Aut(\eta,\xi,\Phi)_0$ be the identity component of the automorphism group.  By the Hilbert-Mumford criterion
one gets that the decomposable element of $\re H^1(\cA)\otimes\C$ are polystable for the action of $G^{\C}$.
If one takes a 1-parameter subgroup $\C^* < G^{\C}$, then we may assume $\U(1)\subset\C^*$ is contained in $G$ by acting by
conjugation.  Consider weight space decomposition
\[ H^1(\cA) =\bigoplus_{k\in\Z} V_k , \]
and $\tau_*(V_k) =V_{-k}$.  Therefore if a decomposable element of $\re H^1(\cA)\otimes\C$ has a nonzero component in
$V_k$ it also does in $V_{-k}$.

\begin{thm}
Suppose $(M,g)$ is a 3-Sasakian manifold and $(\eta,\xi,\Phi)$ is a fixed Sasakian structure.  Then the decomposable elements
of $\re H^1(\cA)\otimes\C$ give integrable Sasaki-Einstein deformations.

In particular, if $H^1(\cA)\neq 0$, then $(M,g)$ admits a non-trivial Sasaki-Einstein deformation.
\end{thm}

In finite families of \emph{toric} 3-Sasakian 7-manifolds were constructed in~\cite{BGMR98}.  These include infinitely many
examples for each second Betti number $b_2(M) =k\geq 1$.  With respect to a fixed Sasakian structure $(\eta,\xi,\Phi)$ if
$b_2(M) >1$ then $\Aut(M,\eta,\xi,\Phi)_0 =T^3$ (cf.~\cite{vanCo12a}).  And one has the following.
\begin{prop}[\cite{vanCo12a}]
Let $(M,g)$ be a toric 3-Sasakian 7-manifold.  Then with respect to a fixed Sasakian structure one has
\[ \dim H^1(\cA) =b_1(M) -1, \]
and $H^1(\cA) =H^1(\cA)^{T^3}$.
\end{prop}

Since $\TT^3 =T^3_{\C}$ acts trivial on $H^1(\cA)$, all the elements are trivially polystable.  Therefore we get a second
proof of a result proved analytically in~\cite{vanCo12b}.
\begin{thm}
Let $(M,g)$ be a toric 3-Sasakian 7-manifold.  Then $(M,g)$ admits a complex $b_2(M)-1$ dimensional space of Sasaki-Einstein
deformations.
\end{thm}

\bibliographystyle{amsplain}

\providecommand{\bysame}{\leavevmode\hbox to3em{\hrulefill}\thinspace}
\providecommand{\MR}{\relax\ifhmode\unskip\space\fi MR }
\providecommand{\MRhref}[2]{%
  \href{http://www.ams.org/mathscinet-getitem?mr=#1}{#2}
}
\providecommand{\href}[2]{#2}

\end{document}